\documentclass[12pt]{amsart}

\usepackage[all,color]{xy}
\usepackage{pb-diagram}
\usepackage[mathscr]{eucal}
\usepackage{hyperref}
\hypersetup{
    colorlinks=true, %set true if you want colored links
    linktoc=all,     %set to all if you want both sections and subsections linked
    linkcolor=blue,  %choose some color if you want links to stand out
}

\usepackage{xcolor}
\usepackage[normalem]{ulem}

\DeclareMathAlphabet{\mathpzc}{OT1}{pzc}{m}{it}

\usepackage{amsfonts}
\usepackage{amsmath}
\usepackage{amssymb}

\newcommand{\tr}{\textnormal{tr}}
\newcommand{\ric}{\textnormal{Ric}}

\newcommand{\dbar}{\overline{\partial}}

\newcommand{\ddbar}{\sqrt{-1}\partial\dbar}

\def\cA{{\mathcal A}}
\def\cB{{\mathcal B}}

\def\cF{{\mathcal F}}

\def\cH{{\mathcal{H}}}

\def\cK{{\mathcal K}}

\def\cJ{{\mathcal J}}

\def\cR{{\mathcal R}}
\def\cS{{\mathcal S}}
\def\cT{{\mathcal T}}
\def\cV{{\mathcal V}}
\def\cX{{\mathcal X}}

\def\cZ{{\mathcal Z}}

\def\PSH{\textnormal{PSH}}

\def\bfq{\mathbf{q}}
\def\bfx{\mathbf{x}}
\def\bfX{\mathbf{X}}
\def\bfY{\mathbf{Y}}

\def\bfy{\mathbf{y}}

\def\bfQ{\mathbf{Q}}

\def\cU{\mathcal{U}}
\def\cW{\mathcal{W}}
\def \Diam{ {\rm Diam}}
\def\RCD{{\rm RCD}}
\def\PSH{{\rm PSH}}

\DeclareMathOperator{\vol}{Vol}
\newcommand{\xk}[1]{\big(#1\big)}
\newcommand{\bk}[1]{\Big(#1\Big)}
\newcommand{\abs}[1]{|#1|^2}

\newtheorem{theorem}{Theorem}[section]

\newtheorem{lemma}{Lemma}[section]
\newtheorem{definition}{Definition}[section]
\newtheorem{corollary}{Corollary}[section]

\newtheorem{conjecture}{Conjecture}[section]

\numberwithin{equation}{section}

\begin{document}

\title{RCD structures on singular K\"ahler spaces of complex dimension three}

\author[{Xin Fu, Bin Guo and Jian Song}
]{Xin Fu $^*$, Bin Guo$^\dagger$ and Jian Song$^{\dagger \dagger}$  }

  \thanks{Xin Fu is  supported by National Key R\&D Program of China 2024YFA1014800
 and NSFC No. 12401073. Bin Guo and Jian Song are supported in part by the National Science Foundation under grants  DMS-2203607 and DMS-2303508, and the collaboration grant 946730 from Simons Foundation.}

\address{$^*$ School of Science, Institute for Theoretical Sciences, Westlake University, Hangzhou 310030, China}

\email{fuxin54@westlake.edu.cn}

\address{$^*$ Department of Mathematics \& Computer Science, Rutgers University, Newark, NJ 07102}

\email{bguo@rutgers.edu}

\address{$^\dagger$ Department of Mathematics, Rutgers University, Piscataway, NJ 08854}

\email{jiansong@math.rutgers.edu}

\begin{abstract}  

{\footnotesize Let $X$ be a projective variety of complex dimension $3$ with log terminal singularities. We prove that every singular K\"ahler metric on $X$ with bounded Nash entropy and Ricci curvature bounded below induces a compact RCD space homeomorphic to the projective variety $X$ itself. In particular, singular K\"ahler-Einstein spaces of complex dimension $3$ with bounded Nash entropy are compact RCD spaces topologically and holomorphically equivalent to the underlying projective variety. Various compactness theorems are also obtained for $3$-dimensional projective varieties with bounded Ricci curvature.  Such results establish connections among algebraic, geometric and analytic structures of klt singularities from birational geometry and provide abundant examples of RCD spaces from algebraic geometry via complex Monge-Amp\`ere equations.}

\end{abstract}

\maketitle

%{\footnotesize \tableofcontents}

\section{Introduction}

Complex Monge-Amp\`ere equations play a central role to study canonical metrics and their geometric applications in K\"ahler geometry after the celebrated solution of Yau to the Calabi conjecture \cite{Y1}. There has been tremendous progress in the past decades with  influx
of new ideas from pluripotential theory,  Riemannian geometry, complex $L^2$-theory, Ricci flow and the minimal model program \cite{Ko1, CC1, CC2, T1, DS1, S2, ST1, ST2} that unravel deep, rich and unifying structures in these fields. The series of papers \cite{S2, FGS, GPT, GPSS1, GPSS2, GPSS3} aim to build a framework that would expand the classical works \cite{Y1, Ko1} on complex Monge-Amp\`ere equations to geometric analysis on complex spaces with singularities. New analytic and geometric estimates are recently established  in \cite{GPSS1, GPSS2,  GS22, GS25} on singular complex spaces based on a very general analytic assumption on the Nash entropy for   Monge-Amp\`ere volume measures associated to  singular K\"ahler metrics. This assumption is satisfied in most geometric settings, particularly in the case when the underlying complex space is a projective variety with log terminal singularities. However, additional geometric assumptions, particularly on curvatures, are required to derive more refined analysis on local singularities and global moduli spaces. The Ricci curvature on a singular complex is defined in the sense of distribution over singularities using the pluripotential theory,  which can be viewed as a complex synthetic notion of Ricci curvature.   One of the main goals in this paper is to establish the equivalence between the positivity notions of Ricci curvature in K\"ahler geometry and in the RCD theory developed in \cite{LV, Stk, AGS, DG1, DG2} on a large class of algebraic varieties with suitable singularities. There have already been striking results \cite{Sz24, GS25, CCHSTT} in this exciting new direction of research.

Let $X$ be an $n$-dimensional normal projective variety equipped with a smooth K\"ahler metric  $\theta_X$ (e.g. pullback of a smooth K\"ahler metric via projective embeddings). A closed positive $(1,1)$-current $\omega$ is said to be a singular K\"ahler metric on $X$ if it is a smooth   K\"ahler metric on the regular part of $X$, i.e., 
$\omega\in C^\infty(X^\circ)$, where 
$$X^\circ=\cR(X)$$ is the regular part of $X$. 
Similar to \cite{GPSS2}, we consider the set 
\begin{equation}\label{vspace}
\mathcal{V} (X, \theta_X, n, A,  p, K)
\end{equation}
 of all singular K\"ahler metrics $\omega$ on $X$ satisfying the following properties. 
\begin{enumerate}

%\item $\omega\in C^\infty(X^\circ)$, where $X^\circ=\cR(X)$ is the smooth part of $X$. 

%\medskip

\item $[\omega]$ is a K\"ahler class on $X$ with  
\begin{equation}\label{cls0}
I_\omega = [\omega]\cdot [\theta_X]^{n-1} \leq A.
\end{equation}
 
\medskip

\item  $p>n$ and 
\begin{equation}\label{nashen0}
\mathcal{N}_{\theta_X, p}(\omega) = \frac{1}{V_\omega}\int_X \left| \log\left( V_\omega^{-1} \frac{\omega^n}{\theta_X^n} \right) \right|^p\omega^n\leq K.
\end{equation}
where $V_\omega = [\omega]^n$ is the volume of $(X, \omega)$.

\medskip

\end{enumerate}
From the entropy bound (\ref{nashen0}),  $\omega$ is a closed positive $(1,1)$-current on $X$ with bounded local potentials by the well-known $L^\infty$-estimate from  \cite{Ko1, EGZ, Zz, GPT}. The singular K\"ahler metric $\omega$ naturally induces a canonical metric measure space introduced in \cite{GPSS2} as in the definition below.

\begin{definition}\label{defmc} Let $X$ be an $n$-dimensional normal projective variety equipped with a smooth K\"ahler metric $\theta_X$. For any $\omega \in \nu(X, \theta_X, n, A, p, K)$ with $p>n$, we define  
\begin{equation}\label{metcom0}
(\hat X, d_\omega ) = \overline{(X^\circ, \omega|_{X^\circ})}
\end{equation}
to be the metric completion of $(X^\circ, \omega|_{X^\circ})$. If we let $\mu_\omega$ be the trivial extension of the smooth volume measure $\omega^n|_{X^\circ}$ to $\hat X$, then  
\begin{equation}\label{metcom1}
(\hat X, d_\omega, \mu_\omega)
\end{equation}
 is defined as the  metric measure  space induced by $(X, \omega)$.

\end{definition}

 The metric measure space $(\hat X, d_\omega, \mu_\omega)$ is extensively studied in \cite{GPSS2} both analytically and geometrically via complex Monge-Amp\`ere equations and the coupled Laplacian equation. 
 In fact, the entropy bound ensures that $(\hat X, d_\omega, \mu_\omega)$ is a compact metric space by the uniform diameter estimates derived in \cite{GPSS2}. The spectral theory is also established in \cite{GPSS2} for $W^{1,2}(\hat X)$ along with many other uniform estimates for Sobolev inequalities, Green's functions and heat kernels.   
They are essential technical preparations for building connections to the RCD theory as well as the general PDE theory on singular complex spaces.  The Ricci curvature for $\omega \in \mathcal{V} (X, \theta_X, n, A,  p, K)$ can be defined as a current if the volume measure $\omega^n$ satisfies suitable positivity condition from pluripotential theory (c.f. Definition \ref{riclbdef}). If the Ricci curvature $\ric(\omega)$ is bounded below, one wishes to establish the geometric and analytic structures of $(\hat X, d_\omega, \mu_\omega)$ beyond the works of \cite{GPSS2} as the tangent cones $(\hat X, d_\omega, \mu_\omega)$ are expected to be unique and algebraic. In particular, one would expect the notion of Ricci curvature bounded below in terms of the pluripotential theory should be equivalent to various synthetic Ricci curvature lower bounds in the study of RCD spaces developed by \cite{LV, Stk, AGS} and many others.

\begin{definition} \label{defak} Let $X$ be an $n$-dimensional  projective variety with log terminal singularities equipped with a smooth K\"ahler metric  $\theta_X$.  We define  
$$\mathcal{RK}(X)$$ to be set of any singular K\"ahler metric $\omega$ on $X$ satisfying 
\begin{enumerate}
\item  $\omega\in \nu(X, \theta_X, n, A, p, K)$ for some $p>n$, $A, K>0$, 
\medskip

\item $  \ric(\omega) \geq  \lambda \omega$ on $X$ as currents  for some $\lambda \in \mathbb{R}$. 

\smallskip

\end{enumerate}
We further define $\mathcal{RK}(n)$ to be the set of $(X, \omega)$, where $X$ is an $n$-dimensional projective variety with log terminal singularities and $\omega\in \mathcal{RK}(X)$.
\end{definition}

We propose the following conjecture on the algebraic and geometric structures of the metric measure space $(\hat X, d_\omega, \mu_\omega)$ induced by $(X, \omega)\in \mathcal{RK}(X)$. 

\begin{conjecture} \label{conj1} For any $(X, \omega) \in \mathcal{RK}(n)$, 
the metric measure space $(\hat X, d_\omega, \mu_\omega)$  induced by $(X, \omega)$ as in Definition (\ref{defmc}) is a compact  $\RCD$ space homeomorphic to the projective variety $X$ itself.

\end{conjecture}

The following is our first main result   toward Conjecture \ref{conj1}.

\begin{theorem} \label{mainthm1} Let $X$ be an $n$-dimensional   projective variety with log terminal singularities equipped with a smooth K\"ahler metric $\theta_X$. Suppose %
\begin{enumerate}

\item there exists a resolution of singularities $\pi: Y \rightarrow X$ such that  the relative anticanonical divisor $-K_{Y/X}$ is  effective, 

\smallskip

\item $\omega\in    \mathcal{RK}(X) $. %
\end{enumerate}
Then the metric measure space $(\hat X, d_\omega, \mu_\omega)$ induced by $(X, \omega)$  is a compact RCD space  satisfying the following. 

\begin{enumerate}

\item $(\hat X, d_\omega, \mu_\omega)$  is  homeomorphic to the projective variety $X$. 

\smallskip

\item   $\cR(\hat X)=X^\circ$  and $\dim_{\mathcal{H}}\cS(\hat X) \leq 2n-3$, where $\cS(\hat X) = \hat X\setminus \cR(\hat X)$ is the singular set of $\hat X$.  

\smallskip

\item The  identity map from $X^\circ$ to itself induces a one-to-one Lipschitz map
$$
\iota: (\hat X, d_\omega) \rightarrow (X, \theta_X).
$$

\end{enumerate}
Furthermore,  if $\ric(\omega)$ is also bounded above,  
$$\dim_{\mathcal{H}}\cS(\hat X) \leq 2n-4. $$ 
\end{theorem}

   There are many examples of projective varieties that satisfy assumption (1) in Theorem \ref{mainthm1}. Among them are those that admit crepant resolutions, i.e.,  $X$ is a normal projective variety that has a resolution of singularities $\pi: Y \rightarrow X$  with $K_{Y/X} = 0$. If $X_{min}$ is an $n$-dimensional smooth minimal model of general type, then the pluricanonical system induces a unique birational morphism
  $$\pi: X_{min} \rightarrow X_{can}$$
  from $X_{min}$ to its unique canonical model $X_{can}$. The pluricanonical map $\pi$ is indeed crepant.  Therefore any singular K\"ahler metric $\omega$ on $X_{can}$ with bounded Nash entropy and Ricci curvature bounded below must induce an RCD space $(X_{can}, d_\omega, \mu_\omega)$.  In particular, Theorem \ref{mainthm1} extends the results in \cite{S2} for geometric characterization of singular K\"ahler-Einstein metrics on $X_{can}$.

   Let $(Z, d, \mu)$ be an $\RCD(\lambda, m)$-space.  For any $p\in Z$,   the volume density of a point $p\in (Z, d, \mu)$  is defined by
 \begin{equation}\label{defnv}
\nu_Z(p) =  \lim_{r\rightarrow 0} \frac{ {\rm Vol}_Z(B_Z(p, r))}{{\rm Vol}_{\mathbb{R}^{m}}(B_{\mathbb{R}^{m}} (0, r))}.
\end{equation} 

Our next result establishes Conjecture \ref{conj1} in complex dimension $3$.

\begin{theorem}\label{mainthm2} Let $X$ be a $3$-dimensional projective variety with log terminal singularities. Then for any singular K\"ahler metric $\omega \in \mathcal{RK}(X)$, the metric measure space $(\hat X, d_\omega, \mu_\omega)$ is a non-collapsed \rm{RCD} space homeomorphic to $X$. Furthermore, 
There exists a universal constant $\epsilon>0$ such that for any $p \in \cS(\hat X)$, 
\begin{equation} \label{volgap}
\nu_{\hat X} (p)\leq 1-\epsilon.
\end{equation}

\end{theorem}

We remark that if the Ricci curvature $\omega$ is also bounded above globally on $X$, then $\dim_{\mathcal{H}}\cS(\hat X) \leq 2$. Theorem \ref{mainthm2} shows that any $(X, \omega) \in \mathcal{RK}(3)$ can be identified as a compact RCD space $(X, d_\omega, \mu)$ that is topologically and holomorphically equivalent to $X$. The significance of Theorem \ref{mainthm2} is reflected by the topological and holomorphic equivalence between the metric structure and algebraic structure of log terminal singularities in complex dimension $3$. Furthermore, the RCD condition for $(\hat X, d_\omega, \mu_\omega)$ immediately gives the optimal exponent for the uniform Sobolev inequality on $X$ in Theorem \ref{mainthm1} and Theorem \ref{mainthm2}, improving the estimates in \cite{GPSS1, GPSS2}.  Refined analysis is developed in \cite{FGS2} for klt singularities in dimension $3$ in terms of algebraiticity and uniqueness for their tangent cones, extending the works of \cite{DS2, LI, LX}.   The volume density gap estimate (\ref{volgap}) is closely related to the volume of klt singularities, where algebraic structures meet the corresponding metric structures.  

Theorem \ref{mainthm2} characterizes all three dimensional K\"ahler-Einstein spaces with log terminal singularities, where the Nash entropy assumption always holds. This is because the unique Kahler-Einstein current  induced by the complex Monge-Ampere equation always has $L^p$ volume measure for some $p>1$ from the assumption on the log terminal singularities. 

\begin{corollary} Suppose $X$ is a $3$-dimensional projective variety with log terminal singularities. Suppose $\omega_{KE}$ is a singular Kahler-Einstein metric on $X$ satisfying
$$\ric(\omega_{KE}) = \lambda \omega_{KE}, ~\lambda\in \mathbb{R}.$$
Then the metric measure space $(\hat X, d_{\omega_{KE}}, \mu_{\omega_{KE}})$ induced by $(X, \omega_{KE})$ is an $\RCD(\lambda, 6)$ space homeomorphic to $X$.

\end{corollary}

Theorem \ref{mainthm2} can also be used to develop various compactness theories for $\mathcal{RK}(3)$.

\begin{definition}Given $n\in \mathbb{Z}^+$, $D$, $v>0$, we define $\mathcal{K}(n, D, v) $
to be the set of  pairs $(X, \omega) \in \mathcal{RK}(n)$ satisfying $\omega\in H^2(X, \mathbb{Z})$  and
\begin{eqnarray}
 -  \omega \leq \ric(\omega) &\leq&  \omega, \\
 \Diam \overline{(X^\circ, \omega)} &\leq& D, \\ 
{\rm Vol}(X, \omega) &\geq& v .
\end{eqnarray}
\end{definition}
We note that $\Diam \overline{(X^\circ, \omega)}  = \Diam (\hat X, d_\omega, \mu_\omega)$ for $n=3$ by Theorem \ref{mainthm2}.  For any $(X, \omega)\in \mathcal{K}(n, D, v)$, we can assume  $\omega\in c_1(L)$ for some holomorphic line bundle $L \rightarrow X$. Let $h$ be the hermitian metric on $L$ with $\ric(h) =\omega$ and $\{\sigma_0, ...., \sigma_{N_k} \} $ be an orthonormal basis of $H^0(X, L^k)$ with respect to 
to the inner product 
$$(\sigma_i, \sigma_j) = \int_X \sigma_i \overline{\sigma_j} h^k (k\omega)^n. $$ The Bergman kernel $\rho_k:  X\rightarrow \mathbb{R}$ associated with $(X, \omega)$
 is defined by
\begin{equation}
\rho_k (x)=  \sum_{j=0}^{N_k} |\sigma_j |^2_{h^k}(x), ~ x\in X.
\end{equation}

The following theorem is a natural extension of Tian's partial $C^0$-estimate \cite{T1} in the fundamental work of  \cite{DS1} using Theorem \ref{mainthm2} and the work in \cite{Zk}. 
 
\begin{theorem} \label{mainthm3} There exist $m=m(D, v) >0$,  $b= b(D, v)>0$ and $B=B(D, v)>0$  such that for any  $k\geq 1$, $(X, \omega) \in \mathcal{K}(3, D, v)$ and $p\in X$ , 
$$ b\leq  \rho_{mk}(p)  \leq B. $$
In particular, every $(X, \omega)  \in \mathcal{K}(3, D, v)$ can be embedded in a fixed projective space $\mathbb{CP}^N$ for some $N=N(D, v)>0$.

\end{theorem}

Based on Theorem \ref{mainthm3}, we can establish compactness theorems for singular K\"ahler-Einstein spaces in complex dimension $3$ with positive or negative scalar curvature (c.f. Theorem \ref{kecomp1} and Theorem \ref{kecomp2}). In particular, the compactness theorem for Fano Kahler-Einstein spaces (Theorem \ref{kecomp1})  provides an analytic approach to for compactifying moduli spaces for $K$-stable $\mathbb{Q}$-Fano threefolds.

%%%%%%%%%%%%%%%%%%%%%%%%%%%%%%%%%%%%%%%%%%%%

\medskip

\section{Ricci curvature for singular K\"ahler metrics}

In this section,  we define the Ricci curvature on a singular K\"ahler space as a current on projective varieties with log terminal singularities. The assumption on log terminal singularities is for convenience as the set-up would generally hold for  normal and $\mathbb{Q}$-Gorenstein varieties.  

Let $X$ be an $n$-dimensional projective variety with log terminal singularities. We implicitly require $X$ is normal and $\mathbb{Q}$-Gorenstein, i.e., $K_X$ is a $\mathbb{Q}$-Cartier divisor. There exists $m\in \mathbb{Z}^+$ such that $K_X^m$ is a Cartier divisor on $X$.  We let $\Omega$ be a smooth adapted volume measure on $X$, i.e., 
$$\Omega = f_U |\sigma|^{2/m}$$ on a local open set $U$ of $X$, where $\sigma$ is a local generator of  $K_X^m$ and $f_U$ is a nowhere vanishing smooth function on $U$. The curvature of $\Omega$
$$\ric(\Omega) = -\ddbar \log \Omega \in -[K_X].$$
is a smooth closed $(1, 1)$-form on $X$ (smooth in the sense of restriction of a smooth form via local holomorphic embeddings of $X$).

\begin{definition} \label{riclbdef} Let $X$ be an $n$-dimensional projective variety with log terminal singularities.  Suppose $\omega \in \mathcal{V}(X, \theta_X, n, A, p, K)$. The Ricci curvature of $\omega$ is said to be bounded below by $\lambda \in \mathbb{R}$ if 
$$- \ddbar \log \frac{\omega^n}{\Omega} +\ric(\Omega) \geq \lambda \omega$$
as currents, i.e., $$- \log\frac{\omega^n}{\Omega} \in \PSH(X, \ric(\Omega) - \lambda \omega). $$
In particular, 
$$\ric(\omega) = - \ddbar\log \omega^n = - \ddbar \log \frac{\omega^n}{\Omega} +\ric(\Omega)$$ is defined to be the Ricci curvature of $\omega$ as a current. 

\end{definition}

 Similarly, we can define singular K\"ahler metrics with Ricci curvature bounded above by $\lambda\in \mathbb{R}$ by requiring 
 $$\log \frac{\omega^n}{\Omega} \in \PSH(X, \lambda\omega - \ric(\Omega)).$$
 
Suppose  $[\omega] $ is a K\"ahler class. Then there exists a smooth K\"ahler metric $\omega_0 \in [\omega]$ with 
$$\omega = \omega_0 + \ddbar \varphi, ~ \varphi\in \PSH(X, \omega_0)\cap L^\infty(X)\cap C^\infty(X^\circ).$$ 
 If $\ric(\omega) \geq \lambda \omega$, 
$\varphi$ would satisfy the following complex Monge-Amp\`ere equation
\begin{equation}\label{mx}
(\omega_0 + \ddbar \varphi)^n = e^{-\lambda \varphi - f} \Omega, 
\end{equation}
for some $f\in C^\infty(X^\circ)$. Straightforward calculations show that
$$\ric(\omega) = \ric(\Omega) + \lambda \ddbar \varphi + \ddbar f \geq \lambda \omega$$
and so 
\begin{equation}\label{maweight}
 f\in \PSH(X,  \ric(\Omega) - \lambda \omega_0).
 \end{equation}
In particular, $f$ is bounded above. Since $\omega \in \nu(X, \theta_X, n, A, p, K)$ for some $p>n$, the potential $\varphi \in L^\infty(X)$ and we immediately have the following lemma. 
\begin{lemma} Let $\omega \in \nu(X, \theta_X, n, A, p, K)$ with $\ric(\omega) \geq \lambda \omega$. There exists $c=c(X, \theta_X, n, A, p, K, \lambda)>0$ such that 
$$\omega^n \geq c\Omega. $$
If $\ric(\omega) \leq \Lambda \omega$,   then there exists $C=C(X, \theta_X, n, A, p, K, \Lambda)>0$ such that 
$$\omega^n \leq C\Omega. $$

\end{lemma}

We also have the following characterization for $\mathcal{RK}(X)$. 

\begin{lemma} The singular K\"ahler metric $\omega\in \mathcal{RK}(X)$ satisfies $\ric(\omega) \geq \lambda \omega$ as currents if and only if
$\ric(\omega) \geq \lambda \omega$ on $X^\circ$ and there exists $c>0$ such that 
$$  \omega^n \geq c\Omega .$$

\end{lemma}

\begin{proof} It suffices to show the only if part. Let $f= - \log \frac{\omega^n}{\Omega}$. Then $\ric(\Omega)  - \lambda\omega_0 +  \ddbar f  \geq 0$ is a closed positive $(1,1)$-current on $X^\circ$ and so $f\in \PSH(X^\circ, (\ric(\Omega) - \lambda \omega_0)|_{X^\circ})$. Since $f$ is bounded above, it can be trivially extended to a function in $\PSH(X, \ric(\Omega) - \lambda \omega_0)$ and the lemma is proved
\end{proof}

Various geometric and analytic estimates are established in \cite{GPSS2} for the metric measure space $(\hat X, d_\omega, \mu_\omega)$ if $\omega\in \nu(X, \theta_X, n, A, p, K, H)$. The following lemma is an immediate corollary by applying the work of \cite{GPSS2}.

\begin{lemma}  Let $\omega \in \nu(X, \theta_X, n, A, p, K)$. Then 

\begin{enumerate}

\item There exists $C=C(X, \theta_X, n, p, A, K)>0$ such that  
$$\Diam(\hat X, d_\omega) \leq C.$$
In particular, $(\hat X, d_\omega, \mu_\omega)$ is a compact metric measure space.

\smallskip 

\item  There exist $q>1$ and $C_S=C_S(X, \theta_X, n,  A, p, K, q)>0$ such that 
$$
\Big(\int_{\hat X} | u  |^{2q}\omega^n   \Big)^{1/q}\le C_S \left( \int_{\hat X} |\nabla u|^2 ~\omega^n + \int_{\hat X} u^2 \omega^n \right) .
$$
for all $u\in W^{1, 2}(\hat X, d, \omega^n)$. 

\smallskip

%\item There exists $C=C(X, \theta_X, n, p, A, K)>0$ such that the following trace formula holds for the heat kernel of $(\hat X, d, \omega^n)$
%
%$$H(x,x, t) \leq \frac{1}{V_\omega} + \frac{C}{V_\omega} t^{-\frac{q}{q-1}}. $$ 

%\medskip

\item Let $0=\lambda_0 < \lambda_1 \leq \lambda_2 \leq ... $ be the increasing sequence of eigenvalues of the Laplacian $-\Delta_\omega$ on $(\hat X, d, \omega^n)$. Then there exists $c=c(X, \theta_X, n, A, p,  K)>0$ such that
$$\lambda_k \geq c k^{\frac{q-1}{q}}. $$

\end{enumerate}

\end{lemma}

Next, we will establish a metric regularization for any $\omega \in \mathcal{RK}$.
Suppose $\omega\in \mathcal{RK}(X)$ with $\ric(\omega) \geq \lambda \omega$. Let 
$$\alpha = \ric(\omega) - \lambda \omega$$
be the positive closed $(1,1)$-current that is continuous in $X^\circ$. Since $[\omega]$ is a K\"ahler class on $X$, we can always replace $\lambda$ by a sufficiently negative constant such that $[\alpha]=-\lambda[\omega] -[K_X] $ is also a K\"ahler class. 

Let $\alpha_0 \in [\alpha] $ be a smooth K\"ahler metric and there exists $\psi \in \PSH(X, \alpha_0)\cap C^\infty(X^\circ)$ such that $$\alpha=\alpha_0 + \ddbar \psi. $$
If we let $\omega_0\in [\omega]$ be a smooth K\"ahler metric, we can further assume the adapted volume form $\Omega$ satisfies
$$\ric(\Omega) = \alpha_0 + \lambda \omega_0.$$ 
Then there exists $\varphi\in \PSH(X, \omega_0)\cap L^\infty(X)\cap C^\infty(X^\circ)$ satisfying
$$(\omega_0 + \ddbar \varphi)^n = e^{-\lambda\varphi -\psi} \Omega, ~\omega=\omega_0 + \ddbar \varphi.  $$

\begin{lemma} \label{pshapp}  There exists a sequence of $\psi_i \in C^\infty(X) \cap \PSH(X, \alpha_0)$ such that $\psi_i \geq \psi$ and $\psi_i$ converges pointwise to $\psi$. Furthermore, $\psi_i$ converges to $\psi$ in $C^\infty(\cK)$ on any  $\cK \subset\subset X^\circ$. 
\end{lemma}

\begin{proof}   We fix a holomorphic embedding $\iota: X \rightarrow \mathbb{CP}^N$ and let $\alpha_0$ tbe the restriction of a smooth K\"ahler metric $\tilde \alpha_0 $ on $\mathbb{CP}^N$. Then for any $\alpha_0$-PSH{} function on $X$, it can be extended to an $\tilde \alpha_0$-PSH function on $\mathbb{CP}^N$ by \cite{CG}. Immediately, there exists $\psi_i \in \PSH(X, \alpha_0) \cap C^\infty(X)$ such that $\psi_i$ converge to $\psi$ decreasingly. Since $\psi$ is smooth on $X^\circ$, $\psi_i$ converges to $\psi$ in $L^\infty(\cK)$ for any $\cK\subset\subset X^\circ$. 

We would like to modify $\psi_i$ so that it converges smoothly on $\cK$. For convenience, we assume $\alpha\in H^2(X, \mathbb{Q})$ and let $L \rightarrow X$ be the ample $\mathbb{Q}$-line bundle with $\alpha \in c_1(L)$. Let $\mathcal{J}$ be the ideal sheaf associated to the singular set $\cS(X)$. Then $mL \otimes \mathcal{J}$ is globally generated for sufficiently large $m\in \mathbb{Z}^+$. Let $h_0$ be the smooth hermitian metric for $L$ with $\ric(h_0) = \alpha_0$ and $\sigma_0, ..., \sigma_{N_m}$ be a basis for $H^0(X, mL\otimes \mathcal{J})$. We let $\phi = \frac{1}{m} \log \left( \sum_{k=0}^{N_m} |\sigma_j|^2_{h_0^m}\right)$. Then $\phi \in \PSH(X, \alpha_0)\cap C^\infty(X^\circ)$ with $\phi$ tending to $-\infty$ near $\cS(X)$.  

Let $\tilde \psi_{i, \epsilon, \delta} = \mathcal{M}_{\epsilon} (\psi_i, \psi + \delta + \delta^2 \phi)$ be the regularized maximum of $\psi_i$ and $\psi+\delta + \delta^2 \phi$  for $\delta, \epsilon>0$ (c.f. \cite{BK}) . By definition, we have $\tilde\psi_{i, \epsilon, \delta}\in \PSH(X, \alpha_0)$ with $\tilde\psi_{i, \epsilon, \delta}\geq \psi_i\geq \psi$ . Furthermore, $\tilde\psi_{i, \epsilon, \delta} \in C^\infty(X)$ since $\tilde \psi_{i, \epsilon, \delta} = \psi_i$ near $\cS(X)$ for sufficiently small $\epsilon>0$. 

For any $K\subset\subset X^\circ$,  $\tilde \psi_{i, \epsilon, \delta}= \psi  + \delta$ on $K$ for sufficiently large $i>0$ and sufficiently small $\delta>>\epsilon$, since $\psi_i$ converges to $\psi$ uniformly in $L^\infty(K)$.  At the same time, $\tilde \psi_{i, \epsilon, \delta} = \psi_i$ near $\cS(X)$ as $\psi+\delta + \delta^2 \phi$ tends to $-\infty$ along $\cS(X)$. By choosing suitable $\epsilon_i, \delta_i \rightarrow 0$, $\tilde \psi_{i, \epsilon_i, \delta_i}$ converges to $\psi$ smoothly on any fixed compact subset of $X^\circ$. This proves the lemma.
\end{proof}

We let $\alpha_i = \alpha_0+ \ddbar \psi_i$. Then $\alpha_i$ is sequence of smooth K\"ahler metrics on $X$ and we can consider the twisted K\"ahler-Einstein equation
\begin{equation} \label{tkeqn}
\ric(\omega_i) = \lambda \omega_i + \alpha_i. 
\end{equation}
Recall that we assume $\lambda <0$.
Then the   complex Monge-Amp\`ere equation equivalent to (\ref{tkeqn})  is given by 
\begin{equation}\label{tkema2}
(\omega_0 + \ddbar \varphi_i)^n = e^{-\lambda\varphi_i  -\psi_i} \Omega, ~\omega_i = \omega_0 + \ddbar\varphi_i.
\end{equation}
From the construction,  $-\psi_i \leq - \psi$  and $-\psi_i$ converges to $\psi$ pointwise. Since $\int_X e^{-\lambda \varphi_i - \psi_i} \Omega = [\omega_0]^n$, $\varphi_i \in \PSH(X, \omega_0)$ is uniformly bounded above.  Hence $ \varphi_i \in \mathcal{V}(X, \theta_X, n, A, p, K')$
 for some  $K'>0$  for all $i>0$.

\begin{lemma} \label{loccon25} There exists $C>0$ such that for all $i>0$, 
\begin{equation}\label{estsec200}
||\varphi_i||_{L^\infty(X)} \leq C, ~\Diam(\hat X, d_{\omega_i)}) \leq C, 
\end{equation}
where $(\hat X, d_{\omega_i}) = \overline{(X^\circ, \omega_i)}. $
Furthermore, if for each $i$, there exists $c_i>0$ such that 
$$\omega_i \geq c_i \theta_X$$ on $X^\circ$ for all $i>0$, then the following hold. 
\begin{enumerate}

\item There exists $C>0$ such that for all $i>0$, we have 
\begin{equation}\label{unidom}
\omega_i \geq C^{-1} \theta_X.
\end{equation}

\item For any compact $\cK\subset\subset X^\circ$ and $m>0$, we have  
\begin{equation}\label{2c3con} 
\lim_{i\to\infty}\|\varphi_i - \varphi\|_{C^m(\cK)} =0. 
\end{equation}

\end{enumerate} 

\end{lemma}

\begin{proof} The estimate (\ref{estsec200}) immediately follows from the results of \cite{GPSS2}. To prove (\ref{unidom}), we will apply the Chern-Lu argument. For simplicity, we assume $\lambda=-1$. Let $D$ be an effective divisor on $X$ whose support contains  $\cS(X)$. Let $\sigma_D$ be the defining section of $D$ and $h_D$ be a smooth hermitian metric for the line bundle associated to $D$. We consider 
 $$H_{i, \epsilon} = \log \tr_{\omega_i}(\theta_X) - A \varphi_i + \epsilon \log |\sigma_D|^2_{h_D}$$
 for small $\epsilon>0$. 
 $H_{i, \epsilon}$ must attain its maximum in $X\setminus D$. Straightforward calculations analogous to the Schwarz lemma in \cite{ST1, ST2} show that
 $$\Delta_{\omega_i} H_{i, \epsilon} \geq  \tr_{\omega_i} (\theta_X) - CA$$
in $X\setminus D$ for a fixed sufficiently large $A>0$ and a uniform constant $C>0$ that are both independent of $i$ and $\epsilon$. The calculations are based on the uniform lower bound of $\ric(\omega_i)$ and the upper bound of holomorphic sectional curvature of $\theta_X$. One can apply the maximum principle to $H_{i, \epsilon}$ since $H_{i, \epsilon}$ is smooth on $X\setminus D$, which leads to the uniform upper bound independent on $i$ and $\epsilon$ for $H_{i, \epsilon}$ . The estimate (\ref{unidom}) is then obtained after letting $\epsilon \rightarrow 0$ since $\varphi_i$ is uniformly bounded. (\ref{unidom}) combined with the equation (\ref{tkema2}) also gives uniform upper bound for $\omega_i$ on any $\cK \subset\subset X^\circ$.

To prove (\ref{2c3con}), we fix any $\cK\subset\subset X^\circ$. Then  $\psi_i$ converges to $\psi$  on $C^{1,\alpha}(\cK)$ and Therefore $\varphi_i$ is uniformly bounded in $C^{3, \gamma}(\cK)$ by (\ref{unidom}) by the standard Schauder estimates and elliptic regularity for linear equations. Since $\psi_i$ converges to $\psi$  in $\cK$ and point-wisely on $X$, we have 
$$\lim_{i\rightarrow \infty} \left\| e^{- \psi_i}  - e^{-\psi} \right\|_{L^1(X, \Omega)} = 0. $$
The stability theorem for complex Monge-Amp\`ere equation \cite{Ko2, DZ} implies $$\lim_{i \rightarrow \infty} \| \varphi_i - \varphi\|_{L^\infty(X)} = 0. $$ Then (\ref{2c3con}) immediately follows. 
\end{proof}

As a consequence, $(X^\circ, \omega_i)$ converges  to $(X^\circ, \omega)$ in $C^{\infty}_{loc}(X^\circ)$.

\section{Almost smooth K\"ahler spaces and RCD spaces}

The following notion of almost smooth metric measure space is introduced in \cite{Ho} to construct and identify new examples to RCD space. A slight modification is also given in \cite{Sz24}). 

\begin{definition} A compact metric measure space $(\cZ, d, \mu)$ is an $m$-dimensional almost smooth metric measure space, if there is an open subset $\cZ^\circ$ of $Z$ satisfying the following properties.

\begin{enumerate}

\item There exist a smooth $m$-dimensional Riemannian manifold $(M, g)$ and a diffeomorphism $\Psi: \cZ^\circ \rightarrow M$ such that $\Psi$ is locally isometric between $(Z^\circ, d)$ and $(M, g)$. 

\item  The restriction of $\mu$ to $\cZ^\circ$ coincides with $m$-dimensional Hausdorff measure $\mathcal{H}^n$ on $Z^\circ$.

\item $\mu(\cZ\setminus \cZ^\circ) =0$ and there exist a sequence $\rho_i\in C^\infty(\cZ^\circ)$ with values in $[0,1]$ such that 
\begin{enumerate}

\item for any $\cK \subset\subset \cZ^\circ$, $\rho_i|_\cK =1$ for all sufficiently large $i$,

\item $\lim_{i\rightarrow\infty} \int_\cZ |\Delta \rho_i| d\mu  < \infty$.

\end{enumerate}

\end{enumerate}

\end{definition}

One of the main purposes in this paper is to identify $(\hat X, \omega, \mu_\omega)$ as an RCD space as an almost smooth metric space associated to $(X^\circ, \omega, \omega^n)$ for a given $(X, \omega) \in \mathcal{RK}(n)$. We will need to construct a family of cut-off functions based on  \cite{St} to characterize the singular set of $X$ by capacity.

\begin{lemma} \label{cutoff} Let $X$ be an $n$-dimesnional projective variety and let $Z$ be a subvariety of $X$ with $ \cS(X) \subset Z$. Suppose $\omega \in \nu(X, \theta_X, n, A, p, K)$ with $p>n$. Then for any $\epsilon>0$ and $\cK\subset\subset X\setminus Z $, there exists  $\rho_\epsilon \in C^\infty(X\setminus Z)$ such that 

\begin{enumerate}

\item $ 0\leq \rho_\epsilon \leq 1$, 
\smallskip

\item $Supp \rho_\epsilon \subset\subset X\setminus Z$, 
\smallskip

\item $\rho_\epsilon=1$ on $\cK$, 
\smallskip

\item $ \int_X |\nabla \rho_\epsilon|^2 \omega^n + \int_X \left|\Delta \rho_\epsilon \right|    \omega^n< \epsilon$, where $\nabla$ and $\Delta$ are the gradient and Laplace operator with respect to $\omega$.

\end{enumerate}

\end{lemma}

\begin{proof}  We pick an ample line bundle $L \rightarrow X$ and let $\cJ_Z$ be the ideal sheaf for $Z$. By replacing $L$ by a sufficiently large power of itself, we can assume that $L\otimes \cJ_Z$ is globally generated. Let $\sigma_1, ..., \sigma_N$ be a basis of $H^0(X, L\otimes \cJ)$ and $h$ be a smooth positively curved hermitian metric for $L$ such that $\theta = \ric(h)$ is a smooth K\"ahler metric on $X$ and $\sum_{i=1}^N |\sigma_i|^2_h  \leq 1$. We define 
$$\eta_\epsilon = \max ( \log\left( \sum_{i=1}^N |\sigma_i|^2_h\right), \log \epsilon) . $$
Obviously, $\eta_\epsilon\in \PSH(X, \theta)$ and $\log \epsilon \leq \eta_\epsilon\leq 0$.  Let $F$ be the standard smooth cut-off function on $[0, \infty)$ with $F=1$ on $[0, 1/2]$ and $F=0$ on $[1, \infty)$. 
Now we let 
$$\rho_\epsilon = F\left(\frac{\eta_\epsilon}{\log \epsilon} \right).$$
Then  $\rho_\epsilon =1 $ on $\cK$ if $\epsilon$ is sufficiently small. Similar to the calculations in \cite{St} (c.f.\cite{S2}) give
\begin{eqnarray*}
&&\int_X \sqrt{-1} \partial \rho_\epsilon \wedge \dbar \rho_\epsilon \wedge \omega^{n-1}  \\
&=&(\log \epsilon)^{-2}  \int_X (F')^2 \sqrt{-1}\partial \eta_\epsilon \wedge \dbar \eta_\epsilon \wedge \omega^{n-1}\\
&\leq& C (\log \epsilon)^{-2} \int_X (-\eta_\epsilon) \ddbar \eta_\epsilon \wedge \omega^{n-1} \\
&\leq& C(\log \epsilon)^{-2} \int_X(-\eta_\epsilon)(\theta+ \ddbar \eta_\epsilon) \wedge \omega^{n-1}
+ C(\log \epsilon)^{-2} \int_X \eta_\epsilon ~ \theta \wedge \omega^{n-1}\\
&\leq& C(-\log \epsilon)^{-1} \int_X(\theta+ \ddbar \eta_\epsilon) \wedge \omega^{n-1} \\
&\leq& C(-\log \epsilon)^{-1} [\theta]\cdot [\omega]^{n-1} \rightarrow 0
\end{eqnarray*}
and 
\begin{eqnarray*}
&&\int_X \sqrt{-1} |\Delta \rho_\epsilon|  \omega^n  \\
&\leq & (-\log \epsilon)^{-1}  \int_X |F'| |\Delta  \eta_\epsilon| \omega^n + (-\log \epsilon)^{-2}  \int_X (F')^2 \sqrt{-1}\partial \eta_\epsilon  \wedge \dbar \eta_\epsilon \wedge \omega^{n-1}\\
&\leq &C(-\log \epsilon)^{-1}  \int_X (\ddbar  \eta_\epsilon+\theta) \wedge\omega^{n-1} + C(\log \epsilon)^{-1}  \\
&\leq& C (-\log \epsilon)^{-1}  \rightarrow 0
\end{eqnarray*}
as $\epsilon \rightarrow 0$.
Therefore we obtain  $\rho_\epsilon \in C^0(X)$ satisfying the conditions in the lemma. The lemma is then proved by smoothing $\rho_\epsilon$.  
\end{proof}

\begin{corollary} Let $X$ be an $n$-dimensional projective variety with log terminal singularities. Suppose $\omega\in \nu(X, \theta_X, n, A, p, K)$ is smooth on $X^\circ$. Then $(\hat X, d_\omega, \mu_\omega)$ is a $2n$-dimensional almost smooth compact metric measure space associated with $X^\circ$.

\end{corollary}

The following connection is established in \cite{Ho} between RCD spaces and almost smooth metric measure spaces. 

\begin{lemma} \label{honda2} Let $(\cZ, d, \mu)$ be an $m$-dimensional almost smooth compact metric measure space associated with an open subset $\cZ^\circ$. Then $(\cZ, d, \mu)$ is a non-collapsed compact  $\RCD(\lambda, m)$-space for some $\lambda\in \mathbb{R}$ if and only if the following hold.

\begin{enumerate}

\item the Sobolev to Lipschitz property holds, 

\item the $L^2$-strong compactness condition holds, 

\item any eigenfunction is Lipschitz, 

\item $\ric(g) \geq \lambda g$ on $\cZ^\circ$, where $g$ is the Riemannian metric associated to $d$ on $\cZ^\circ$.

\end{enumerate}

\end{lemma}

Lemma \ref{honda2} can be applied to singular K\"ahler spaces based on the spectral theory built in \cite{GPSS2}.

\begin{lemma} \label{honda3} Let $X$ be an $n$-dimensional normal projective variety with log terminal singularities. Suppose $\omega \in \nu(X, \theta_X, n, A, p, K)$ with $p>n$ and 
$$\ric(\omega) \geq \lambda \omega.$$ Then the metric space $(\hat X, d_\omega, \mu_\omega)$ is a non-collapsed compact $\RCD(\lambda, 2n)$ space if and only if each eigenfunction of $\Delta_\omega$ is Lipschitz.  In particular, the Sobolev inequality holds on $(\hat X, d_\omega, \mu_\omega)$ with the optimal exponent $\frac{n}{n-1}$ if the latter holds.

\end{lemma}

\begin{proof}  Condition (4) holds by the assumption on the Ricci curvature lower bound of $\omega$ and the fact that $\omega$ is smooth on $X^\circ$. Condition (1) and (2) automatically hold for $(\hat X, d_\omega, 
\mu_\omega)$ by the spectral theory established in \cite{GPSS2} for any singular K\"ahler metric $\omega\in \nu(X, \theta_X, n, A, p, K)$. It is proved in \cite{GPSS2} that if $f$ is an eigenfunction of $(\hat X, d_\omega, \mu_\omega)$, then $f\in W^{1,2}(\hat X, d_\omega, \mu_\omega)\cap L^\infty(\hat X)\cap C^\infty(X^\circ)$.   
Then by Lemma \ref{honda2}, $(\hat X, d_\omega, \mu_\omega)$ is indeed an $\RCD(\lambda, 2n)$-space if and only if each eigenfunction is Lipschitz.
\end{proof}

%%%%%%%%%%%%%%%%%%%%%%%%%%%%%%%%%%%%%%%%%%%%

\section{Analytic $L^2$-estimates}

In this section, we will assume that $X$ is an $n$-dimensional  projective variety with log terminal singularities and $\omega \in \mathcal{RK}(X) \cap c_1(L)$ for a line bundle $L \rightarrow X$ with 
$$\ric(\omega) \geq - \omega.$$ Then there exists a hermitian metric $h$ for $L$ with
$$\omega = \ric(h).$$

The following $L^2$-estimate for $\dbar$-equation is a direct generalization for singular K\"ahler varieties.

\begin{lemma} \label{L241}For any $k\geq 2$ and smooth $L^k$-valued $(0,1)$-form $\tau$ satisfying 
\begin{enumerate}

\item $\dbar \tau =0$,

\smallskip

\item $Supp\, \tau \subset \subset X^\circ$, 

\smallskip

\end{enumerate}
there exists an $L^k$-valued section $u$ such that %
\begin{equation}
\bar \partial u = \tau
 \end{equation}
on $X^\circ$ and
$$\int_X |u|^2_{h^k} (k\omega)^n \leq \frac{1}{4\pi} \int_X |\tau|^2_{h^k, k\omega} (k\omega)^n. $$

\end{lemma}

\begin{proof} The lemma is proved by a trick of Demailly \cite{De} (c.f. \cite{S2}). We first pick any ample divisor $D$ on $X$ whose support contains $\mathcal{S}(X)$. Then $X \setminus D$ is a Stein manifold. There exists a smooth plurisubharmonic function $\psi$ on $X\setminus D$ such that $\theta_D = \ddbar \psi$ is a complete K\"ahler metric on $X\setminus D$. Then $$\omega_\epsilon = k\omega+ \epsilon \theta_D$$ is obviously also a complete K\"ahler metric. We now view $\tau$ as an $L^k\otimes K_{X\setminus D}^{-1}$ valued $(n,1)$-form on $X\setminus D$ and let 
$$h_\epsilon = h^ke^{- \epsilon \psi}\omega^n$$
be a hermitian metric on $L^k\otimes K_{X\setminus D}^{-1}$ satisfying
$$\ric(h_\epsilon) = k\omega + \epsilon \theta_D + \ric(\omega)  \geq (k-1)\omega + \epsilon \theta_D  \geq \frac{1}{2} \omega_\epsilon. $$ 
Therefore by Demailly's $L^2$-estimate for $\dbar$-equation, there exist an $L^k\otimes K_{X\setminus D}^{-1}$-valued $(n,0)$-form $u_\epsilon$ satisfying
$$\bar \partial u_\epsilon = \tau$$
on $X\setminus D$ and
$$\int_{X\backslash D} |u_\epsilon|^2_{h_\epsilon} \omega_\epsilon^n \leq \frac{1}{4\pi} \int_{X\backslash D} |\tau|^2_{h_\epsilon, \omega_\epsilon} \omega_\epsilon^n.$$
We now let $\epsilon \rightarrow 0$. Since $u_{\epsilon_1} - u_{\epsilon_2}$ is holomorphic with uniform $L^2$-bound on a fixed compact subset in $X\setminus D$, it must be uniformly bounded and therefore $u_\epsilon$ converges to an $L^k$-valued  section $u$ satisfying 
$\dbar u = \tau$ on $X\setminus D$ and $$ \int_X |u|^2_{h^k} (k\omega)^n \leq \frac{1}{4\pi} \int_X |\tau|^2_{h^k, k\omega} (k\omega)^n. $$
Since $\omega$ is smooth on $X^\circ$, the $L^2$-bound for $u$ implies that $\dbar u = \tau$ holds on $X^\circ$. We now have completed the proof.
\end{proof}

\begin{lemma} \label{L2411} Suppose  $(\hat X, d_\omega, \mu_\omega)$ is a compact $\RCD(-1, 2n)$ space and there exists $c>0$ such that 
$$\omega\geq c \theta_X .$$
 Then  for any $\sigma \in H^0(X, L^k)$, we have
\begin{equation}\label{inibdsig}
\sup_{X^\circ} |\sigma|^2_{h^k}   < \infty,
\end{equation}
\begin{equation}\label{inibdsid2}
  \sup_{X^\circ} |\nabla \sigma|^2_{h^k, k\omega}   < \infty.
\end{equation}

\end{lemma}

\begin{proof} Let $h_0$ be a smooth hermitian metric of $L$ such that $\omega_0= \ric(h_0)>0$ is a smooth K\"ahler metric. Then $h= h_0 e^{-\varphi}$ and $\omega= \omega_0 + \ddbar\varphi$ for some $\varphi\in \PSH(X, \omega_0)\cap L^\infty(X)\cap C^\infty(X^\circ)$. Then $|\sigma|^2_{h^k} = |\sigma|^2_{h_0^k} e^{-k \varphi}$ and (\ref{inibdsig}) follows immediately. 

%For any $\sigma\in H^0(X, L^k)$,  $ h_0 \partial \sigma \wedge \dbar \overline{\sigma} \leq C\omega_0$ for some $C>0$. Then  $$\int_X |\nabla \sigma|^2_{h^k, \omega} \omega^n = \int_X h_0^ke^{-k\varphi} \partial \sigma \wedge \dbar \overline{\sigma} \wedge \omega^{n-1} \leq C\int_X \omega^n  $$
%
%since $\varphi \in L^\infty(X)$.
%
Let $\iota: (\hat X, d_\omega) \rightarrow (X, \theta_X)$ be the extension of the identity map on $\hat X$. Then $\iota$ is  Lipschitz by our assumption. For any $p\in \hat X$, let $q= \iota(p)$ and choose an open neighborhood $U\subset \hat X$ of $p$, such that $B_{\theta_X}(q, r) \supset \iota(U)$ for some sufficiently small $r>0$. Then we can view $\sigma$ as a holomorphic function on $U\cap X^\circ$. In particular, $\Delta_\omega \sigma =0$ on $U\cap X^\circ$. Both of the real and imaginary parts of $\sigma$ can be extended to bounded harmonic functions on $U$. We also note that $\omega$ is a polarized metric of $L$, so the local potential $\phi$ of $\omega$ satisfies that $\Delta\phi=n$. Therefore $|\nabla \sigma|^2_{h, \omega}\leq|\partial \sigma|^2_\omega h_0^k e^{-k\varphi}+|\sigma|^2_h|\nabla\phi|^2_\omega$ is bounded by the gradient estimate for harmonic functions (also $\phi$) on an RCD space \cite{J14} as the generalization of Cheng-Yau's gradient estimate. We have now completed the proof of the lemma. 
\end{proof}

 Direct calculations show that 
\begin{equation}\label{lapsig41}
\Delta_{k\omega}  |\sigma|^2_{h^k} = \mathrm{tr}_{k\omega}\left( \left(\nabla \sigma \wedge  \overline{\nabla \sigma} \right) h^k \right) -  n|\sigma|^2_{h^k}  \geq - n|\sigma |^2_{h^k}.
\end{equation}
At any $p\in X^\circ$, by choosing normal coordinates at $p$ with respect to $\omega$, we can assume that $\nabla h(p) = 0$. Direct calculations  show that 
\begin{equation}\label{lapsig42}
\Delta_{k\omega} |\nabla \sigma|^2_{h^k, k\omega} \geq |\nabla^2 \sigma|^2_{h^k, k\omega} - (1+k^{-1})n |\nabla \sigma|^2_{h^k, k\omega} +n  |\sigma|^2_{h^k}
\end{equation}
by the Bochner formula.

We denote $\Delta_\sharp$ and $\nabla_\sharp$ by the Laplace and gradient operators with respect to the rescaled metrics $h^k$ and $k\omega$ and similarly define the scaled norm $\| \cdot \|_{L^{\infty, \sharp}}$ and $\| \cdot \|_{L^{2, \sharp}}$ for $s\in H^0(X, L^k)$ with respect to the hermitian metric $h^k$ and $k\omega$.

\begin{lemma} \label{L242} Suppose  $(\hat X, d_\omega, \mu_\omega)$ is a compact $\RCD(-1, 2n)$ space. There exists  $\Lambda >0$ such that if  $s\in H^0(X, L^k)$ for $k\geq 1$, then 
\begin{equation}\label{l2form1}
\|s\|_{L^{\infty, \sharp}} \leq \Lambda \|s\|_{L^{2, \sharp}}.
\end{equation}
\begin{equation}\label{l2form2}
 \|\nabla_\sharp s\|_{L^{\infty, \sharp}} \leq \Lambda \|s\|_{L^{2, \sharp}}. 
\end{equation}
\end{lemma}

\begin{proof} The estimates (\ref{lapsig41}) and (\ref{lapsig42}) imply that 
$$\Delta_\sharp |\sigma|_{\sharp} \geq - n |\sigma|_{\sharp}, ~ \Delta_\sharp |\nabla_\sharp \sigma|_\sharp \geq - n |\nabla\sigma|_\sharp, $$
where   $| \cdot |_\sharp$ are taken with respect to the rescaled metrics $h^k$ and $k\omega$.  Since $\cS(X)$ is an analytic subvariety of $X$ and has $0$-capacity, we can apply the same argument in \cite{S2} (c.f. Proposition 3.3 and Proposition 3.4) by combining Moser's iteration, the cut-off functions  for $|\sigma|_\sharp$ and $|\nabla_\sharp \sigma|$ in Lemma \ref{cutoff},  (\ref{inibdsig}) and (\ref{inibdsid2}) in Lemma \ref{L242} to  obtain the $L^2$-estimates (\ref{l2form1}) and (\ref{l2form2}). 
\end{proof}

%%%%%%%%%%%%
%%%%%%%%%%%%%%%%%%%%%%%%%%%%%%%%%%%%%%%%%%%%

\section{Proof of Theorem \ref{mainthm1}: a special case}

In this section, we will prove a special case for Theorem \ref{mainthm1}. We consider a special family of projective varieties that admit a resolution of singularities with effective relative anti-canonical bundle.

Let $X$ be an $n$-dimensional projective normal variety of log terminal singularities. We assume that $X$ admits a resolution of singularities 
$$\pi: Y \rightarrow X$$ 
with 
\begin{equation}\label{spres} 
K_Y = \pi^*K_X - \sum_{i=1}^I a_i E_i, ~a_i\geq 0, i=1, ..., I, 
\end{equation}
where $E_j$ is a prime divisor. Since $X$ has log terminal singularities, $a_i \in (-\infty, 1)$. Throughout this section we will assume $-K_Y$ is $\pi$-effective, i.e., the relative anti-canonical bundle $-K_{Y/X}$ is effective. Therefore
$$0\leq a_i < 1, ~\leq i=1, ...,  I.$$
 The following is the main result of this section. The singular set $\cS(X)$  coincides with the support of $\cup_{i=1}^I\pi(E_i)$ since $a_i\geq 0$, otherwise there must be some $a_i<0$.  The following is the main result of this section. 
 
\begin{theorem} \label{thm3spe} Suppose $X$ is an $n$-dimensional projective variety with log terminal singularities that admits a resolution of singularities satisfying (\ref{spres}). Let $\omega\in \nu(X, \theta_X, n, A, p, K)$ with $p>n$ be  a twisted K\"ahler-Einstein metric $\omega$  on $X$ satisfying
\begin{equation}\label{tkespe}
\ric(\omega) = - \omega + \alpha
\end{equation}
for a smooth non-negative closed $(1,1)$-form $\alpha$. Then the metric measure space $(\hat X, d_\omega, \mu_\omega)$ induced by $(X^\circ, \omega, \omega^n)$ is a non-collapsed $\RCD(-1, 2n)$ space satisfying the following properties. 
\begin{enumerate}

\item  $(\hat X, d_\omega)$ is homeomorphic to the original variety $X$.

\medskip

\item $\mathcal{R}(\hat X)= X^\circ$.

\medskip

\item  There exists $c>0$ such that $$\omega \geq c \theta_X. $$

\item There exists $\epsilon=\epsilon(n)>0$ such that for any $p \in \cS(\hat X)$, 
$$
\nu_{\hat X} (p)< 1-\epsilon.
$$

\end{enumerate}

\end{theorem}

 The rest of the section is devoted to the proof of Theorem \ref{thm3spe}.
 
 \subsection{Approximating metrics on $Y$} We let $\omega$ be the singular K\"ahler metric in the assumption of Theorem \ref{thm3spe}. We choose a smooth reference K\"ahler metric $\omega_X\in [\omega]$. Then there exists $\varphi \in \PSH(X, \omega_X) \cap L^\infty(X) \cap C^\infty(X^\circ)$. 
$$\omega = \omega_X + \ddbar \varphi.$$  
Since $\alpha \in [\omega] - [K_X]$ is smooth, there exists 
 a smooth adaptive volume measure $\Omega_X$  satisfying 
 $$\alpha = \ric(\Omega_X) + \omega_X.$$

After adding a constant to $\varphi$, we can view $\varphi$ as the solution to the following complex Monge-Amp\`ere equation
\begin{equation}\label{ma1}
(\omega_X + \ddbar \varphi)^n = e^{\varphi } \Omega_X. 
\end{equation}
which is equivalent to the corresponding curvature equation (\ref{tkespe}).

Let 
$$E'= \cup_{a_i>0} E_i, ~E''=\cup_{a_i=0} E_i, ~ E=E'\cup E'' $$
and let $\sigma_i$ be the defining section for $E_i$ respectively. We choose a smooth   volume measure   $\Omega_Y$ on $Y$ and  there exist smooth hermitian metrics $h_i$ associated to divisors $E_i$ such that
$$\Omega_Y = \prod_{i=1}^I |\sigma_i|^{2 a_i}_{h_i} \pi^* \Omega_X . $$
By Kodaira lemma, there exists a $\mathbb{Q}$-effective divisor $D$ with 
$\pi^*[\omega_X] - [D]>0$ and so we can further choose a smooth hermitian metric $h_D$  associated to  $[D]$ such that  
\begin{equation}\label{defthy}
\theta_Y = \pi^*\omega_X - \ric(h_D)>0
\end{equation}
is a smooth K\"ahler metric on $Y$. If we let $\sigma_D$ be the defining section of $D$, then 
$\theta_Y= \pi^*\omega_X + \ddbar \log |\sigma_D|^2_{h_D}$ on $X\setminus D$.
We can further assume that the support of $D$ contains that of the exceptional divisor $E$.

We will approximate equation (\ref{ma1}) by   the following families of complex Monge-Amp\`ere equations 

\begin{equation}\label{ma3}
(\pi^*\omega_X + \delta_k \theta_Y + \ddbar \varphi_k )^n = e^{ \varphi_k} \prod_{i=1}^I |\sigma_i|_{h_i}  ^{-2a_i}\Omega_Y, 
\end{equation}
and  
\begin{equation}\label{ma4}
(\pi^*\omega_X + \delta_k  \theta_Y + \ddbar \varphi_{k,  j} )^n = e^{ \varphi_{k,  j}} \prod_{i=1}^I \left( |\sigma_i|^2_{h_i} + \epsilon_j \right)^{-a_i}\Omega_Y 
\end{equation}
with paremeters $\delta_k\in (0,1)$ and $\epsilon_i\in (0, 1)$ satisfying  
$$\lim_{k\rightarrow \infty} \delta_k = \lim_{j\rightarrow \infty} \epsilon_j = 0. $$
We let 
\begin{equation}
\omega_k = \pi^*\omega_X + \delta_k \theta_Y + \ddbar \varphi_k , ~\omega_{k,j} = \pi^*\omega_X + \delta_k \theta_Y + \ddbar \varphi_{k,j}. 
\end{equation}
It is well-known that both equation (\ref{ma3}) and (\ref{ma4}) admit unique solutions in $L^\infty(Y)$ as the volume measures are $L^p$ for some $p>1$. Furthermore, $\varphi_{k, j}$ is smooth for all $k, j>0$ and $\varphi_k$ is smooth on $Y\setminus E'$ as shown in \cite{GPSS2}. One can further obtain uniform estimates for both $\varphi_{k, j}$ and $\varphi_k$ as below.

\begin{lemma} \label{linf55} There exists $C>0$ such that for all $j, k>0$,  
$$||\varphi_{k, j} ||_{L^\infty(Y)} + \|\varphi_k\|_{L^\infty(Y)}\leq C. $$

\end{lemma}

\begin{proof} By integration on both hand sides and Jensen's inequality, we have uniform upper bound for $\int_Y \varphi_{k, j} \Omega_Y$. Since $\varphi_{k, j} \in \PSH(Y, \pi^*\omega_X+ \delta_k \theta_Y)$, $\varphi_{k, j}$ is uniformly bounded above.  Similarly, $\varphi_k$ is also uniformly bounded above.  $L^\infty$-estimate of \cite{Ko1} and its extensions  for complex Monge-Amp\`ere equations directly leads to the conclusion of the lemma since the volume measures on the right hand sides of (\ref{ma3}) and (\ref{ma4}) are uniformly bounded in $L^p(Y, \Omega_Y)$ for some fixed $p>1$.
\end{proof}

The uniform $L^\infty$-estimates also imply the uniform volume bounds, i.e., there exists $C>0$ such that for all $j, k>0$,
\begin{equation}\label{vol55bd}
C^{-1} \leq \vol(Y, \omega_{k, j}^n) \leq C, ~ C^{-1} \leq \vol(Y, \omega_k^n) \leq C. 
\end{equation}

The $L^\infty$-estimates in Lemma \ref{linf55} immediately imply a uniform diameter bound for $(Y, \omega_{k, j})$ and $(Y_k, d_{\omega_k})$ by the results of \cite{GPSS1, GPSS2}, where 
$$(Y_k, d_{\omega_k}) = \overline{(Y\setminus E', \omega_k)}.$$

\begin{lemma} \label{dia55} There exists $C>0$ such that for all $k, j>0$, we have
$$\Diam(Y, \omega_{k, j}) \leq C, ~ \Diam(Y_k, d_{\omega_k}) \leq C. $$

\end{lemma}

The Ricci lower bounds for $\omega_{k, j}$ and $\omega_k$ are given in the following lemma.

\begin{lemma} There exists $C>0$ such that for all $j, k >0$, we have 
\begin{equation}\label{riclb52}
\ric(\omega_{k, j} ) \geq - Ce^{C\delta_k^{-1}} \omega_{k, j}
\end{equation}
on $Y$ and
\begin{equation}\label{riclb51}
\ric(\omega_k) \geq -  \omega_k. 
\end{equation}
on $Y\setminus E'$.

\end{lemma} 
\begin{proof} By direct calculations, we have 
$$\ric(\omega_k) = - \omega_k + \pi^* \alpha + \delta_k\theta_Y \geq - \omega_k.$$ 
Straightforward calculations show that there exists $C>0$ such that for all $i=1, ..., I$ and $j>0$ 
$$\ddbar \log (|\sigma_i|^2_{h_i} + \epsilon_j) \geq - C \theta_Y $$
since  $\log |\sigma_i|^2_{h_i} \in \PSH(Y, A\theta_Y)$ for some sufficiently large $A>0$. 
Therefore
\begin{eqnarray*}
&& \ric(\omega_{k,  j})  \\
&=& - \omega_{k, j} + \delta_k\theta_Y + \sum_{i=1}^I a_i \ric(h_i) + \sum_{i=1}^I a_i \ddbar \log (|\sigma_i|^2_{h_i} + \epsilon_j) \\
&=& -   \omega_{k, j} - C   \theta_Y
\end{eqnarray*}
for some uniform $C>0$. In order to bound $\theta_Y$ by $\omega_{k, j}$, we let $H = \log tr_{\omega_{k, j}}(\theta_Y)  - 2B \delta_k^{-1} \varphi_{k, j}$ for some fixed sufficiently large $B>0$. Direct calculations give 
$$\Delta_{\omega_{k, j}} H \geq  B tr_{\omega_{k, j}} (\theta_Y) - CB.$$
 We can apply the maximum principle to $H$ and obtain a uniform estimate 
$$tr_{\omega_{k, j}}(\theta_Y) \leq C e^{2\delta_k^{-1} \|\varphi_{k, j}\|_{L^\infty(Y)}} \leq Ce^{C\delta_k^{-1}}$$
for some uniform $C>0$. Then (\ref{riclb52}) immediately follows. 
\end{proof}
 
The proof of (\ref{riclb52}) also produces a global lower bound for $\omega_{k,j}$ with
\begin{equation}\label{metlb51}
\omega_{k, j} \geq C^{-1} e^{- C\delta_k^{-1}} \theta_Y,
\end{equation}
which leads to uniform second order estimates for $\varphi_{k, j}$.Using suitable barrier functions, one can further prove uniform local second order estimates for $\varphi_{k, j}$ on $X^\circ$. Combined with Schauder estimates and linear PDE theory, we have the following regularity result.

\begin{lemma} \label{highreg1}  For any $\cK \subset \subset Y\setminus E'$,  $m>0$ and $k>0$, there exists $C=C_{\cK, m, k}>0$ such that for all $j>0$, 
$$\|\varphi_{k,  j} \|_{C^m(\cK)} \leq C.$$
\end{lemma}

 Using suitable barrier functions (c.f. \cite{ST1, S4}), one can further obtain uniform second order estimates for $\varphi_{k, j}$ and $\varphi_k$ on any fixed $\cK\subset\subset Y\setminus E'$ for all $k, j>0$. This would lead to the following uniform estimates independent of $k$.
 
\begin{lemma} \label{highreg2}  For any $\cK \subset \subset Y\setminus E'$ and $m>0$, there exists $C=C_{\cK, m}>0$ such that for all $j, k>0$, 
$$\|\varphi_{k,  j} \|_{C^m(\cK)} \leq C, ~||\varphi_k ||_{C^m(\cK)} \leq C.$$
\end{lemma}

The above estimates imply that $\varphi_{k, j}$ converges smoothly to $\varphi_k$ on any compact subset of $Y\setminus E'$ as $j\rightarrow \infty$ and  $\varphi_k$ converges smoothly to $\varphi$ on any compact subset of $Y\setminus E$. 
 We also have the following Schwarz lemma for $\omega_k$. 
 \begin{lemma} \label{schw3} There exists $c>0$ such that for all $k$, 
\begin{equation}\label{omgklb5}
\omega_k \geq c \pi^*\theta_X. 
\end{equation}
In particular, there exists $C>0$ such that for all $k>0$, 
\begin{equation}\label{ricul5}
  \ric(\omega_k) \geq - \omega_k
\end{equation}
on $Y\setminus E'$. 

\end{lemma}

\begin{proof} By (\ref{metlb51}), for any fixed $k>0$, there exists $c_k>0$ such that  
$$\omega_{k, j} \geq c_k \pi^*\theta_X. $$
By letting $j\rightarrow \infty$, we can conclude that $\omega_k \geq c_k \theta_Y$ on $Y \setminus E'$ since $\omega_{k, j}$ converges smoothly to $\omega_k$ on $Y\setminus E'$. 
 We let
$$H= \log \mathrm{tr}_{\omega_k}( \pi^*\theta_X) - A \varphi_k + \varepsilon \log |\sigma_D|^2_{h_D} $$
for some uniform constant $A>0$, where $D$, $\sigma_D$ and $h_D$ are defined earlier for $\theta_Y$ in (\ref{defthy}). We can apply the maximum principle to $H$ because $H$ tends to $-\infty$ near $E'$. Then 
$$\Delta_{\omega_k} H \geq c \mathrm{tr}_{\omega_k}( \pi^* \theta_X) - C$$
for some uniform $C>0$.  Since $\varphi_k$ is uniformly bounded in $L^\infty(Y)$, the maximum principle would infer a uniform upper bound $C>0$ independent of $k$ for 
$$H \leq C $$
for  all $0 <\varepsilon <<1$, 
Then (\ref{omgklb5}) is proved by letting $\varepsilon \rightarrow 0$. (\ref{ricul5}) follows immediately by (\ref{omgklb5}) since
$$  \ric(\omega_k) = - \omega_k + \pi^*\alpha + \delta_k \theta_Y \geq -\omega_k. $$ 
We have completed the proof of the lemma.
\end{proof}

\subsection{The RCD spaces $(Y_k, d_k, \mu_k)$}

For fixed $k>0$, $(Y, \omega_{k, j}, \omega_{k, j}^n)$ is a sequence of smooth K\"ahler manifolds with Ricci curvature uniformly bounded below by the bounds in (\ref{riclb52}) for $j=1, 2, ...$ We can immediately apply the Cheeger-Colding theory to obtain the Gromov-Hausdorff limit of $(Y, \omega_{k, j}, \omega_{k, j}^n)$.

\begin{lemma}  \label{ykspace51} For each $k>0$, $(Y, \omega_{k, j}, \left( \omega_{k, j}\right)^n)$ converges in Gromov-Hausdorff topology to a non-collapsed compact $\RCD(-1, 2n)$ space $( Y_k,  d_k, \mu_k)$ as $j \rightarrow \infty$. In particular, the convergence is smooth on $Y\setminus E'$ with 
\begin{equation}\label{reginclu57}
 Y\setminus E' \subset \cR (Y_k)
 \end{equation}
  and 
\begin{equation}\label{convex57}
( Y_k, d_k) =   \overline{(Y\setminus E', \omega_k)} .
\end{equation}

 \end{lemma}

\begin{proof} For each fixed $k$, $\ric(\omega_{k, j})$ is uniformly bounded below by $-Ce^{\delta_k^{-1}}$ by (\ref{riclb52}) and the diameter of $(Y, \omega_{k, j})$ is uniformly bounded above for all $j>0$. Therefore $(Y, \omega_{k, j}, \omega_{k,j}^n)$ converges in Gromov-Hausdorff sense to an $\RCD(-Ce^{\delta_k^{-1}}, 2n)$ space $( Y_k, d_k, \mu_k)$, after possibly passing to a subsequence, by the standard Cheeger-Colding theory. 

Lemma \ref{highreg1} implies smooth convergence of  $\omega_{k, j}$ to the K\"ahler metric $\omega_k$ on $Y\setminus E'$ and so (\ref{reginclu57}) holds .

 Since $E'$ is an analytic subvariety of Hausdorff codimension no greater than $2$, by the Gromov's trick \cite{Gr}, $Y\setminus E'$  is uniformly almost convex in $Y$ with respect to $\omega_{k, j}$ (c.f. Lemma 2.2 \cite{STZ}). In particular, the metric $d_k$ is induced by $\omega_k$ on $Y\setminus E'$ with $Y\setminus E'$ being almost convex in $Y_k$. This proves (\ref{reginclu57}) and (\ref{convex57}). 

Finally, we would like to verify that  $( Y_k, d_k, \mu_k)$ is an $\RCD(-1, 2n)$ space, i.e., the synthetic Ricci curvature is globally bounded below by $-1$. By the smooth convergence of $\omega_{k, j}$ on $Y\setminus E'$, $(Y_k, d_k, \mu_k)$ is an almost smooth compact  metric measure space associated with the open smooth subset $Y\setminus E'$.  On the other hand, $(Y_k, d_k, \mu_k)$ is already an RCD space,  therefore every eigenfunction of $(Y_k, d_k, \mu_k)$ must be Lipschitz.  We can apply Lemma \ref{honda3}  combined with the estimate (\ref{riclb51}) to  conclude that $( Y_k, d_k, \mu_k)$ is indeed an $\RCD(-1, 2n)$ space.
\end{proof}

\begin{lemma} The identity map from $Y\setminus E' \subset Y_k \rightarrow Y\setminus E' \subset Y$ extends to a surjective Lipschitz map 
$$
\iota_k: (Y_k, d_k) \rightarrow (Y, \theta_Y).
$$
\end{lemma}

\begin{proof} The lemma follows immediately from the estimate (\ref{metlb51}). 
\end{proof}

\begin{lemma} \label{1to1-51} The map $\iota_k$ is one-to-one. In particular, the metric space $Y_k$ is   homeomorphic to $Y$ and the complex structure of $Y\setminus E'$ extends to that of $Y_k$.

\end{lemma}

\begin{proof} For any point $p\in Y$, we can choose an open neighborhood $U$ of $p$ such that the line bundle $L$ is trivial on $U$. In fact, by (\ref{metlb51}), $U$ must contain an open geodesic ball in $(Y_k, d_k)$ centered at any point in $\iota_k^{-1}(p)$. For each fixed $k>0$, $(U, \omega_{k, j})$ is a sequence of open K\"ahler manifolds with Ricci curvature uniformly bounded below for all $j>0$. We can apply the work of \cite{DS1, LS1, LS2} combined with the geometric $L^2$-estimates Lemma \ref{L242} and Lemma \ref{L241} for $\omega_{k, j}$ and the smooth convergence of $\omega_k$ on $Y\setminus E'$. For any two distinct points $p_1$ and $p_2$ in $\iota_k^{-1}(U)$, one can construct a holomorphic function $\sigma$ on $U$ satisfying the following properties.
\begin{enumerate}
\item $\pi^*\sigma$  extends to a holomorphic function on $\iota_k^{-1}(U)$.

\item $\pi^*\sigma(p_1) \neq \pi^*\sigma(p_2).$

\end{enumerate}
Therefore $\iota_k(p_1) \neq \iota_k(p_2)$ and so $\iota_k$ is one-to-one. Immediately $Y_k$ is homeomorphic to $Y$ and we can extend the complex structure of $Y\setminus E'$   to $Y_k$ trivally.
\end{proof}
Lemma \ref{1to1-51} identifies $Y_k$ with the projective variety $Y$ in both topological and complex structures.

\begin{lemma} \label{angl} 
For any point $p$ in the smooth part of $E_i \subset E' $ (i.e. $a_i> 0$), the tangent cone of $( Y_k, d_k)$ at $p$ is the flat $\mathbb{C}^{n-1} \times \mathbb{C}_{2(1-a_i) \pi}$, where $\mathbb{C}_\gamma$ is the flat $\mathbb{C}$ with cone angle $\gamma$. 
In particular, $$\nu_{Y_k}(p) = (1-a_i). $$

\end{lemma}

\begin{proof} By the Schauder estimates developed in \cite{D2, CDS2, GP, GS21} for complex Monge-Amp\`ere equations with conical singularities, the tangent cone of $p$ in the smooth part of  $E_i$ with $a_i>0$ must be $\mathbb{C}^{n-1} \times \mathbb{C}_{2(1-a_i) \pi}$. Therefore $\nu_{Y_k} (p) = 1- a_i<1$. 
\end{proof}

%By Lemma \ref{highreg1}, $\omega_k$ is smooth along the smooth part of $E''$. 

Lemma \ref{angl} can be immediately extended to everywhere in $E'$.

\begin{lemma} \label{lem38} For any $p\in  E'$, 
$$\nu_{Y_k}(p) \leq 1- \min_{i\in I, a_i>0}  a_i. $$

\end{lemma}

\begin{proof} Since $Y_k$ is homeomorphic to $Y$, $p\in E'$ can be approximated by points in the smooth part of $E'$ in Gromov-Hausdorff distance.  The lemma then follows immediately from volume comparison for RCD spaces and Lemma \ref{angl}. 
\end{proof}

\begin{corollary} $\cR(Y_k) = Y\setminus E'$.

\end{corollary}
 
\begin{proof} The proof directly follows from Lemma \ref{ykspace51} and Lemma \ref{lem38}.
\end{proof}

%%%%%%%%%%%%%%%%%%%%%%%%%%%%%%%%%%%%%%%%%%%

\medskip

 \subsection{The RCD space $(Y_\infty, d_\infty, \mu_\infty)$}

We now  consider the sequence of the non-collapsed $\RCD(-1, 2n)$ spaces $(Y_k, d_k, \mu_k)$. Thanks to  the uniform diameter bound for $(Y_k, \omega_k)$ (Lemma \ref{dia55}) and volume bounds (\ref{vol55bd}), we obtain a limiting RCD space 
$$(Y_\infty,  d_\infty, \mu_\infty) $$ from $\{(Y_k, d_k, \mu_k)\}_{k=1}^\infty $ as $k \rightarrow \infty$, by the general compactness theory of non-collapsed RCD spaces. In particular, the convergence is smooth on $X^\circ =Y\setminus E$ by Lemma \ref{highreg2} and so
$$X^\circ \subset \mathcal{R}(Y_\infty). $$
Lemma \ref{schw3} implies that  the identity map from $X^\circ$ to itself extends to a  surjective Lipschitz map
\begin{equation}\label{iota2}
\iota_\infty: (Y_\infty, d_\infty) \rightarrow (X, \theta_X).
\end{equation}

\begin{lemma} $(Y_\infty, d_\infty)$ is isometric  to $(\hat X, d_\omega)= \overline{(X^\circ, \omega )}$. 

\end{lemma} 

\begin{proof}   The lemma follows from the same argument in Lemma \ref{ykspace51}  almost convexity of $X^\circ$ via Gromov's trick as in Lemma \ref{ykspace51}. 
\end{proof}

The above lemma establishes the equivalence of $(Y_\infty, d_\infty, \mu_\infty)$ and $(\hat X, d_\omega, \mu_\omega). $
\begin{corollary}

$$(Y_\infty, d_\infty, \mu_\infty)=(\hat X, d_\omega, \mu_\omega). $$

\end{corollary}

\begin{lemma} \label{dengap} There exists $\epsilon>0$ depending on $X$ such that for any $p\in Y_\infty \setminus X^\circ$, 
$$\nu_{Y_\infty} (p) <  1- \epsilon. $$

\end{lemma}

\begin{proof} Let $q= \iota_\infty(p)$ for any fixed $p\in Y_\infty \setminus X^\circ$. Then $q\in \cS(X)$. Suppose $\nu_{Y_\infty} (p)\geq 1- \epsilon$ for some sufficiently small $\epsilon>0$ to be determined later. There exist $r>0$ and $p_k \in (Y_k, d_k)$ such that $p_k \rightarrow p$ in Gromov-Hausdorff distance and for all sufficiently large $k$,
$$\frac{{\rm Vol}(B_{Y_k}(p_k, r))}{{\rm Vol} ( B_{\mathbb{R}^{2n}}(0,r))} > 1- 2\epsilon$$
by the volume convergence for non-collapsed RCD spaces.

 By Cheeger-Colding's extension of the Reifenberg's theorem, if $\epsilon$ is sufficiently small, $d_{GH}(B_{Y_k}(p_k, r), B_{\mathbb{R}^{2n}}(0, r))$ must be sufficiently small and $B_{Y_k}(p_k, r)$ is homeomorphic to the Euclidean unit ball $B(0, 1)$ in $\mathbb{R}^{2n}$.  By fixing a sufficiently small $\epsilon$, $E'\cap B_{Y_k}(p_k, r) = \emptyset$ and $\omega_k$ must be smooth in $B_{Y_k}(p_k, r)$ for all sufficiently small $r>0$ because of Lemma \ref{lem38}. Therefore $\omega_k = \ddbar \psi_k$ in $B_{Y_k}(p_k, r)$  for some bounded and smooth PSH function $\psi_k$ because we can view $\omega_k$ as the curvature of the trivial line bundle $L_k$ equipped with the hermitian metric $e^{-\psi_k}$. Applying Proposition 2.4 in \cite{LS1} for sufficiently small $\epsilon>$, there exists a holomorphic chart for $B_{d_\infty}(p, r)$ which induces a holomorphic map 
 $$F: B_{d_\infty}(p, r) \rightarrow B_{\mathbb{C}^n}(0, 1) \subset \mathbb{C}^n.$$
 Furthermore, $F$ is a homemorphism between $B_{d_\infty}(p, r)$ and its image by $F$.  In particular, $Y_\infty$ is locally a smooth open K\"ahler manifold near $p$ with possibly a singular metric structure.

We now choose an open neighborhood $U_q$ of $q$ such that  
$$\iota_\infty|_{B_{d_\infty}(p,r)}: B_{d_\infty}(p, r) \rightarrow U_q$$
 is a Lipschitz map due to  Lemma \ref{schw3}. By the same argument in Lemma \ref{1to1-51}, $\iota_\infty|_{B_{d_\infty}(p, r)}$ is injective. Therefore $\iota_\infty|_{B_{d_\infty}(p, r)}$ is locally biholomorphic with $q\in \iota_\infty(B_{d_\infty}(p, r))$. However, $q$ is an analytic singularity while $B_{d_\infty}(p, r)$ is smooth. This leads to contradiction  and the lemma is then proved. 
 \end{proof}

We immediately can characterize the regular set of $Y_\infty$.

\begin{corollary}\label{regp1} $ \mathcal{R}(Y_\infty) =X^\circ $. 

\end{corollary}

%By (\ref{ricul5}) in Lemma \ref{schw3}, the Ricci curvature of $\omega$ is uniformly bounded on $\cR(Y_\infty)$. 

 It is proved in \cite{DG2} as a generalization of the Cheeger-Colding theory \cite{CC1, CC2} that for any $m$-dimensional non-collapsed RCD space $Z$, any iterated tangent cone at any point $z\in Z$ are metric cones and the singular set $\cS(Z)$ admits a stratification $\cS_0(Z) \subset \cS_1(Z) \subset ... \subset \cS_{m-1} (Z)\subset Z$, where $\cS_k(Z)$ is the set of points where no tangent cone splits off an isometric Euclidean factor $\mathbb{R}^{k+1}$.  

\begin{lemma} $\cS_{2n-1}(Y_\infty) =\cS_{2n-2}(Y_\infty) $.

\end{lemma}

\begin{proof} Suppose $p\in \cS_{2n-1}(Y_\infty) \backslash \cS_{2n-2}(Y_\infty) $. Then there exist a sequence of points $p_k \in Y_k$ and $r_k>0$ with $\lim_{k\rightarrow \infty} r_k =0$ such that $$d_{GH}(B_{(Y_k, r_k^{-1} d_k)}(p_k, 1), B_{\mathbb{R}^{2n-1}\times \mathbb{R}^+}(0, 1) )\rightarrow 0. $$ By the $\epsilon$-regularity for boundary points of RCD spaces in \cite{BNS}, there will be a point $q_k \in \cS_{2n-1}(Y_k)\backslash \cS_{2n-2}(Y_k)$. However, $\cS_{2n-1}(Y_k)\backslash \cS_{2n-2}(Y_k) =\emptyset$ by the Cheeger-Colding theory since $Y_k$ arises as the Gromov-Hausdorff limit of manifolds with Ricci curvature bounded below. This leads to contradiction and the lemma is proved. 
\end{proof}

Similarly, $\cS_{2n-1}\backslash \cS_{2n-2}$ is empty on any iterated tangent cone of $(Y_\infty, d_\infty)$.  
Suppose $(Z, d,\mu)$ is an $n$-dimensional RCD space. The $\varepsilon$-regular set of $Z$ is given by  
 $$\mathcal{R}_\varepsilon (Z)= \{ p\in Z: \nu_Z (p) > 1- \varepsilon\}. $$

%Since $\iota_\infty: Y_\infty \rightarrow X$ is Lipschitz and $X^\circ = \cR(Y_\infty)$, any holomorphic section $\sigma \in H^0(X, L^{k'})$ can be extended as a Lipschitz section on $Y_\infty$ due to estimate (\ref{inibdsid2}). 

\begin{lemma} \label{lsana} Let $(V, o)$ be a tangent cone at $(Y_\infty, p)$. Suppose the $\varepsilon$-singular set $V\setminus \mathcal{R}_\varepsilon(V)$ has $0$ capacity for some $\varepsilon>0$.  Then,  there exist $k_0, C>0$ such that for some $m\leq k_0$, there exists $\sigma \in H^0(X, L^m)$ satisfying 
\begin{enumerate}

\item $\|\sigma\|_{L^2(h^{}, mk\omega)} \leq C. $

\smallskip

\item $\left| |\sigma(z)| - e^{-mkd_{\infty}(z, p)} \right| < \xi. $

\end{enumerate}

\end{lemma}

\begin{proof} The lemma is a consequence of Proposition 3.1 of \cite{LS1} with slight modifications (or \cite{DS1} since the Ricci curvature of $\omega$ is bounded both above and below on $\cR(Y_\infty)$). If $Y_\infty$ is the Gromov-Hausdorff limit of uniformly non-collapsed polarized manifolds with uniform lower Ricci bound, the lemma would immediately follow from Proposition 3.1 of \cite{LS1}. However, $Y_\infty$ is the limit of $(Y_k, d_k, \mu_k)$ and one has to take care of the singularities of $Y_k$. 

Suppose $(V, o)$ is the pointed Gromov-Hausdorff limit of $(Y_\infty, p_i, A_i d_\infty)$ associated to the line bundle $L_i = A_i L$ with $A_i \rightarrow \infty$. Let 
$$\Sigma = \{ x\in V: ~ \nu_V (x) \leq 1- \varepsilon \} $$
for a fixed sufficiently small $\varepsilon>0$. 
Obviously, $\Sigma$ is closed by continuity of the tangent cones of a non-collapsed RCD space \cite{CN, Den}. Let $\Sigma_\delta$ be the $\delta$-neighborhood of $\Sigma$ and $U_{\delta, R}=B(o, R) \setminus \Sigma_\delta$. By choosing sufficiently small $\delta>0$, we can assume that  the open set $U_{\delta, R}$ does not contain any limit of singular points of $X$ by Lemma \ref{dengap},. 

In fact, $U_{\delta, R}$ is the Gromov-Hausdorff limit of a sequence of open smooth K\"ahler manifolds in $X^\circ$ with Ricci curvature  bounded below by $-1$ (and uniformly bounded above (c.f. (\ref{ricul5})). Let $U_i$ be lift of $U_{\delta, R}$ back in $X^\circ \subset Y_k$ under Gromov-Hausdorff approximation. One can follow exactly the same argument in the proof of Proposition 3.1 (\cite{LS1}) based on the geometric $L^2$ estimates in Lemma \ref{L242} and Lemma \ref{L241} to construct holomorphic charts on $U_i$ and then apply the partial $C^0$-techniques with suitable cut-off functions to prove the lemma. 
\end{proof}

Lemma \ref{lsana} also holds for local holomorphic sections instead of global sections in $H^0(X, L^k)$ because $\iota_\infty$ is Lipschitz.

\begin{lemma} \label{sinn-3} For any point  $p\in \cS(Y_\infty)$, $p \in \cS_{2n-3}(Y_\infty)$. 

\end{lemma}

\begin{proof} Suppose the tangent cone of $p$ is $V=\mathbb{R}^{2n-2}\times C_\gamma$, where $C_\gamma$ is the cone over $S^1$ with angle $\gamma<2\pi$.   Let  $g_V = \frac{1}{2} \nabla^2 r^2$ be  the cone metric on $V$ with $r$ being the distance function from the vertex of $V$. There is a natural complex structure $J$ on the regular part of $V$ by \cite{LS1} induced by the blow-up sequence of $(Y_\infty, d_\infty)$. One can show $J$ extends to the complex structure on $V$ with the splitting factor $\mathbb{R}^{2n-2}$ identified as $\mathbb{C}^{n-1}$ as below. Let $(x_1, x_2, ..., x_{2n-2})$ be the Euclidian coordinates for $\mathbb{R}^{2n-2}$ and $(\rho, \theta)$ be the polar coordinates for $C_\gamma$.  Let $w_1 = \nabla x_1 = \frac{\partial}{\partial x_1} $ and  $v_1= J w_1= J \nabla x_1$.  Since $J$ is parallel with respect to $g_V$, we have $\nabla v_1 = J\nabla^2 x_1 = 0. $
We now let $y_1 = \frac{1}{2} \langle v_1, \nabla r^2 \rangle. $
Then
$$2 \nabla y_1 = \langle \nabla v_1,  \nabla r^2 \rangle + \langle v, \nabla^2 r^2 \rangle =2 \langle v_1, Id\rangle = 2v_1, ~ \nabla^2 y_1 = \nabla v_1 = 0. $$ 
Therefore $y_1$ is Lipschitz and harmonic on $\cR(V_p)$. Immediately, $y$ extends to a global harmonic function with respect to $g$ as a harmonic coordinate function. In particular,  $v_1$ and $w_1$ span a splitting factor $\mathbb{C}$ in $\mathbb{R}^{2n-2}$. We can replace $x_{n+1}$ by $y_1$ and repeat the above procedure turning $\mathbb{R}^{2n-2}$ to $\mathbb{C}^{n-1}$.

Obviously, the singular set of $V$ has $0$ capacity and we can apply Lemma \ref{lsana}. Since $\cS(V)$ has $0$-capacity and $\cS(V)$ contains all the limiting points from $\cS(Y_\infty)$ by Lemma \ref{dengap}, one can obtain $n$ Gaussian sections as in Proposition 9 and Proposition 10 \cite{CDS2} to approximate the holomorphic coordinates $[z]=(z_1, z_2, ..., z_n)$ on $V$ with $\epsilon$-Kahler embeddings from any $W \subset \subset \{ [z]< r\} \subset V$ for  $r>0$. These Gaussian sections are global $L^2$ holomorphic sections constructed on $X$ by Lemma \ref{L242} and Lemma \ref{L241}. They induces a holomorphic map 
$F: X \rightarrow \mathbb{C}^n$. Let $\cB\subset X$ be the closure of  $\{ x \in X^\circ: d_\infty(x, p) < 1, ~F(x) <1\}$ in $X$. Then the argument for Proposition 12 \cite{CDS1} gives a holomorphic equivalence from $\cB$ to the unit ball in $\mathbb{C}^n$. On the other hand, $\iota_\infty(p) \in \cB$ is a singular point of $X$, $\cB$ cannot be biholomorphic to a unit ball in $\mathbb{C}^n$. This leads to contradiction. Therefore $\cS_{2n-2}(Y_\infty) = \emptyset$ and we have proved the lemma.
\end{proof}

Similarly, for any iterated tangent cone of $(Y_\infty, d_\infty)$, either it is $\mathbb{C}^n$ or it can only split off $\mathbb{R}^k$ with $k< 2n-2$ because the Ricci curvature is uniformly bounded on $\cR(Y_\infty)$.

\begin{corollary}  \label{sinn-32} For any point  $p\in \cS(Y_\infty)$, the singular set of the tangent cone at $p$ is closed and has Hausdorff dimension  is no greater  than $2n-3$. In particular, the capacity of the singular set of any tangent cone is $0$.

\end{corollary}

\begin{lemma} $(Y_\infty, d_\infty)$ is homeomorphic to  $X$. More precisely, the map $\iota_\infty$ is a one-to-one Lipschitz map. 

\end{lemma}

\begin{proof} By Lemma \ref{sinn-3}, the singular set of any tangent cone of $Y_\infty$ must have $0$ capacity. Therefore we can apply Lemma \ref{lsana} and construct holomorphic Gaussian sections based on the geometric $L^2$-estimates (c.f. Lemma \ref{L242} and Lemma \ref{L241}). Then the lemma is proved by the same argument for Lemma \ref{1to1-51}.
\end{proof}

\begin{lemma} \label{dengap55} There exists $\epsilon(n)>0$ such that   for any $p\in \cS(Y_\infty)$, 
$$\nu_{\hat X}(p) <  1 - \epsilon. $$

\end{lemma}

\begin{proof} Suppose the lemma fails. Then there exist $\epsilon_j \rightarrow 0$, a sequence of $n$-dimensional projective varieties $(X_j, \omega_j)$ satisfying the assumptions of Theorem \ref{thm3spe} and a sequence of $p_j \in \cS(X_j)$ with 
$$\nu_{X_j}(p_j) = 1 - \epsilon_j \rightarrow 1. $$
Let $(\hat X_j, \hat d_j, \hat \mu_j)$ be the RCD space induced by $(X_j, \omega_j, \omega_j^n)$. Then from what we have proved earlier, $\hat X_j$ is homeomorphic to $X_j$ and can be identified as the projective variety $X_j$ itself. By the volume comparison, $\hat X_j$ is uniformly non-collapsed and $(\hat X_j, p_j, \hat d_j)$ converged in pointed Gromov-Hausdorff topology to a non-collapsed RCD space $(\hat X_\infty, p_\infty, \hat d_\infty, \hat \nu_\infty)$. Let $(V_j, o_j)$ be a tangent cone of $p_j \in \hat X_j$. We can assume that after passing to a subsequence, $(V_j, o_j)$ converge in pointed Gromov-Hausdorff topology to $(V, o)$, the tangent cone of $p_\infty$.   We will discuss in the following cases.

\begin{enumerate}

\item Suppose $\cS(V_j)=\{o_j\}$ for all sufficiently large $j$ after passing to a subsequence.  Then the link of $(V, o)$ must be the standard $S^{2n-1}$ as the volume density is equal to $1$. In particular, the convergence is smooth away from the vertices of these cones. Therefore the links of $(V_j, o_j)$ converge smoothly to standard $S^{2n-1}$. By the differential sphere theorem, the links of $(V_j, o_j)$ must be isometric to standard $S^{2n-1}$ for sufficiently large $j$. This leads to contradiction as it would imply that $\nu_{X_j}(p_j) = 1$ for sufficiently large $j>0$. 

\smallskip

\item Suppose $(V_j, o_j)=\mathbb{R}^{2n-2}\times W_j$  for all large $j$ after passing to a subsequence. Then $W_j=\mathcal{C}_{\gamma_j}$ is the cone with cone angle $0<\gamma_j < 2\pi$. One can apply the partial $C^0$-estimate as in the proof of Lemma \ref{sinn-3} and show that $p_j$ must be a smooth point instead of a complex singularity. This leads to contradiction. 

\smallskip

\item Suppose $(V_j, o_j)=\mathbb{R}^k \times W_j$ for some $0<k<2n-2$ and for all large $j$ after passing to a subsequence. By slightly choosing a singular point $p_j'$ of $W_j$ away from $o_j$ and applying rescaling of $(V_j, p_j)$, one can increase the dimension $k$ of the flat factor. Repeating this procedure, one can eventually return to case (2) and obtain contradiction.

\end{enumerate}
\end{proof}

Combining the above results, we have completed the proof of  Theorem \ref{thm3spe}.

%%%%%%%%%%%%%%%%%%%%%%%%%%%%%%%%%%%%%%%%%%%%

\section{Proof of Theorem \ref{mainthm1}: the general case}

We will prove Theorem \ref{mainthm1} in this section by proving the following theorem as a generalization of Theorem \ref{thm3spe}.

\begin{theorem} \label{thm3gene} Suppose $X$ is an $n$-dimensional projective variety with log terminal singularities that admits a resolution of singularities satisfying (\ref{spres}). Let $\omega\in \nu(X, \theta_X, n, A, p, K)$ with $p>n$ be a singular K\"ahler metric  on $X$ satisfying
\begin{equation}\label{ricbelow}
\ric(\omega)  + \omega \geq 0
\end{equation}
in the sense of currents. Then the metric measure space $(\hat X, d_\omega, \mu_\omega)$ induced by $(X^\circ, \omega, \omega^n)$ is a non-collapsed $\RCD(-1, 2n)$ space satisfying the following properties. 
\begin{enumerate}

\item  $(\hat X, d_\omega)$ is homeomorphic to the original variety $X$.

\medskip

\item  $\mathcal{R}(\hat X)= X^\circ$.

\medskip

\item  There exists $c>0$ such that $$\omega \geq c \theta. $$
%
%\medskip

\item There exists $\epsilon=\epsilon(n)>0$ such that for any $p \in \cS(\hat X)$, 
$$
\nu_{\hat X} (p)<  1 -\epsilon.
$$

\end{enumerate}
Furthermore,  $\ric(\omega)$ is also bounded above, then 
$$\dim_{\mathcal{H}}\cS(\hat X) \leq 2n-4. $$ 

\end{theorem}

   We will begin our proof of Theorem \ref{thm3gene}.  We first let $$\alpha = \ric(\omega) + \omega \geq 0,$$ 
which is not necessarily smooth as considered in Theorem \ref{thm3spe}. 
 We can assume that $[\alpha]$ is a K\"ahler class. If not,  we can let $\omega' = A \omega $ and replace $\alpha$ by $\alpha'$ with 
 $$\alpha' = \ric(\omega') + \omega' = \ric(\omega)  + A\omega$$ lies in a K\"ahler class for sufficiently large $A>0$.  
We choose a smooth K\"ahler metric  $\alpha_0 \in [\alpha]$ and let $\alpha = \alpha_0 + \ddbar \psi$ for some $\psi\in \PSH(X, \alpha_0)\cap C^\infty(X^\circ)$. By Lemma \ref{pshapp}, there exist a sequence $\psi_j\in \PSH(X, \alpha_0)\cap C^\infty(X)$ that $\psi_j$ converge to $\psi$. 
If we let $\alpha_j = \alpha_0 + \ddbar \psi_j$ and let $\omega_j \in [\omega]$ be the unique singular K\"ahler metric satisfying 
$$\ric(\omega_j)= - \omega_j + \alpha_j$$
with bounded Nash entropy as constructed in (\ref{tkema2}). Then we have the following results from Lemma \ref{loccon25}.

\begin{lemma}\label{loccon66}
$\omega_j$ converges to $\omega$ in $C^\infty(\cK)$ for any $\cK\subset\subset X^\circ$ as $j\rightarrow \infty$ with uniform bounds on their diameters and Nash entropy bounds.  %
 \end{lemma}

 Obviously $\omega_j$ satisfies the assumptions of Theorem \ref{thm3spe} and so $(\hat X_j, d_{\omega_j}, \mu_{\omega_j})$ is an $\RCD(-1, 2n)$ space homeomorphic to $X$. After possibly taking a subsequence, we can assume that $(\hat X_j, d_{\omega_j}, \mu_{\omega_j})$ converge to an $\RCD(-1, 2n)$ space $(X_\infty, d_\infty, \mu_\infty)$ due to the uniform diameter bound and volume bounds for $(\hat X_j, d_{\omega_j}, \mu_{\omega_j})$ . In particular, we have $X^\circ \subset \cR(X_\infty)$ by Lemma \ref{loccon66}. 

  \begin{lemma} \label{sch66} There exists $c>0$ such that 
 $$\omega\geq c\theta_X. $$
 
 \end{lemma}
 
 \begin{proof} We will apply the Schwarz lemma in Lemma \ref{schw3}, where the constant $c>0$ in (\ref{omgklb5}) is independent of the choice $\alpha$. Therefore $\omega_j \geq c \theta_X$ for all $j>0$ and lemma follows by letting $j\rightarrow$ since $\omega_j$ converges to $\omega$ in $C^\infty_{loc}(X^\circ)$.
 \end{proof}
 We let $$\iota_\infty: X_\infty \rightarrow X$$ be the extension of the identity map from $X^\circ$ to itself. Then by Lemma \ref{sch66}, $\iota_\infty$ is a surjective Lipschitz map.

% Let $$\cR_\epsilon(X_\infty)=\{ p\in X_\infty: \nu_{X_\infty}(p) > 1- \epsilon\}. $$

 %
\begin{lemma} \label{epsin63} Let $\epsilon=\epsilon(n)>0$ be the constant in Lemma \ref{dengap55}. Then
\begin{equation}
\cR(X_\infty) = \cR_\epsilon(X_\infty) = X^\circ.
\end{equation}

\end{lemma}

\begin{proof} For any $p\in X_\infty\setminus X^\circ$, there exist $p_j \in X_j \setminus X^\circ$ with $p_j \rightarrow p$. By the volume convergence and Lemma \ref{dengap55}, there exists a universal $\epsilon=\epsilon(n)>0$ such that 
$$\nu_{X_\infty}(p) < 1 - \epsilon. $$ 
Therefore $p$ cannot be a regular point of $X_\infty$ and so $\cR(X_\infty) = X^\circ$. The lemma is proved by applying Lemma \ref{dengap55} again to $\cR_\epsilon(X_\infty)$. 
\end{proof}
By continuity of the tangent cones \cite{CN, Den} and Lemma \ref{epsin63}, $X^\circ$ is convex in $X_\infty$ and so $$(X_\infty, d_\infty, \mu_\infty) = (\hat X, d_\omega, \mu_\omega). $$

The following lemma is identical to Lemma \ref{lsana} with the same proof.

 \begin{lemma} \label{lsana2} Let $(V, o)$ be a  tangent cone $(X_\infty, p)$ at $p$. Suppose the $\varepsilon$-singular set $V\setminus \mathcal{R}_\epsilon(V)$ has $0$ capacity for some $\varepsilon>0$.  Then for any given small $\xi>0$, there exist $k_0, C>0$ such that for some $m\leq k_0$, there exists $\sigma \in H^0(X, L^m)$ satisfying 
\begin{enumerate}

\item $\|\sigma\|_{L^2(h^{}, mk\omega)} \leq C. $

\smallskip

\item $\left| |\sigma(z)| - e^{-mkd_{\infty}(z, p)} \right| < \xi. $

\end{enumerate}

\end{lemma}

Similar to Lemma \ref{sinn-3}, we have the following lemma. 
\begin{lemma} \label{split8} Any tangent cone of $(X_\infty, d_\infty, \mu_\infty)$ cannot split off $\mathbb{R}^{2n-1}$ or $\mathbb{R}^{2n-2}$. In particular, $\dim_\cH \cS(X_\infty) \leq 2n-3.$

\end{lemma}
 
%Similar to Lemma \ref{sinn-3} and Corollary \ref{sinn-32}, we have the following lemma. 
 
%\begin{lemma} For any point  $p\in \cS(X_\infty)$, the singular set of the tangent cone at $p$ is closed and has Hausdorff dimension  is no greater  than $2n-3$. In particular, the capacity of the singular set of any tangent cone is $0$. 

%\end{lemma}

If $\omega$ has bounded Ricci curvature, then the Hausdorff codimension of the singular sets of any iterated tangent cone must be no less than 4.

\begin{lemma}\label{lemcod4} Suppose $\ric(\omega)$ is further bounded above. Then $\cS(X_\infty) = \cS_{2n-4}$ and $$\dim_{\mathcal{H}} \cS(X_\infty) \leq 2n-4. $$

\end{lemma}

\begin{proof} Suppose there exist $p\in \cS(X_\infty)$ and a tangent cone $V$ at $p$ with $V=\mathbb{R}^{2n-3}\times W$. By the Ricci curvature upper bound of $X_\infty$ and Lemma \ref{epsin63}, $W$ is Ricci flat on its regular part. We will discuss in the following cases for the metric cone $(W, o_W)$.

\begin{enumerate}

\item Suppose $\cS(W) = \{o_W\}$. The link of $\cS(W)$ must be a smooth positively curved Einstein manifold of real dimension $2$. This implies that it must be the standard $S^2$ and so $W = \mathbb{R}^3$. Contradiction. 

\smallskip

\item Suppose $\cS(W)\neq \{o_W\}$. Then we can choose a point $p\in \cS(W)\setminus\{o_W\}$ and after rescaling, the corresponding iterated tangent cone will split off $\mathbb{R}^{2n-2}$.  This contradicts the fact that $\cS_{2n-2}=\emptyset$ for any iterated tangent cones.

\end{enumerate}
Therefore $\cS_{2n-3} (X_\infty) =\emptyset$ and we have completed the proof of the lemma.
\end{proof}

\begin{lemma} \label{0cap8} There exists $\epsilon>0$ such that for any  tangent cone $(V, o)$ of $(X_\infty, d_\infty, \mu_\infty)$,  $V\setminus \cR_\epsilon(V)$ has $0$ capacity. 
\end{lemma}

\begin{proof} We let $S_1 \subset \cS(V)$ be the closure of the set of limiting points from $\cS(X_\infty)$. Then $S_1$ is closed and $S_1 \subset \cS_{2n-3}(V)$ by Lemma \ref{epsin63} and Lemma \ref{split8}. Therefore $S_1$ has capacity $0$. For any $\varepsilon>0$, $R>>1$ and any compact subset $\cK$ of $\cR_\epsilon (V)$ for sufficiently small $\epsilon>0$ to be determined later, there exists a cut-off function $\rho_1$ such that $\rho_1=1$ on $\cK$ and $\rho_1$ vanishes in an open neighborhood $U_1$ of $S_1$ satisfying
$$\|\nabla \rho_1\|_{L^2(B_V(o, R))} < \varepsilon^2. $$
Since $V\setminus U_1$ is the limit of regular points in $X^\circ$, it is the Gromov-Hausdorff limits of non-collapsed open K\"ahler manifolds with Ricci curvature uniformly bounded below. Then the results of \cite{LS1} imply that $S_2 = (V\setminus U_1) \setminus \cR_\epsilon (V)$ has $0$ capacity by \cite{LS1} for some small but uniform $\epsilon>0$. In fact, $V\setminus \cR_\epsilon(V)$ is a smooth K\"ahler manifold and $S_2$ is contained in an analytic subvariety of $V\setminus U_1$.  By the choice of $U_1$ and $\cK$,  there exists a cut-off function $\rho_2$ with support in $V\setminus U_1$ such that $\rho_2=1$ on $\cK$, $\rho_1$ vanish in an open neighborhood of $ (V\setminus U_1) \setminus \cR_\epsilon (V)$ and 
$$\|\nabla \rho_2\|_{L^2(B_V(o, R)\setminus U_1)} < \varepsilon^2. $$
Then $\rho=\rho_1\rho_2$ is the cut-off function such that $\rho=1$ on $\cK$, $\rho$ vanishes in an open neighborhood of $V \setminus \cR_\epsilon (V)$ and 
$$\|\nabla \rho\|_{L^2(B_V(o, R))} < \varepsilon. $$
This completes the proof of the lemma.
\end{proof}

Immediately, we can strengthen Lemma \ref{lsana} by removing the the $0$-capacity assumption on the singular set of the tangent cone.

\begin{corollary}\label{lsana88} Let $(V, o)$ be any tangent cone at $(X_\infty, p)$. Then for any given small $\xi>0$,  there exist $k_0, C>0$ such that for some $m\leq k_0$, there exists $\sigma \in H^0(X, L^m)$ satisfying 
\begin{enumerate}

\item $\|\sigma\|_{L^2(h^m, m\omega)} \leq C. $

\smallskip

\item $\left| |\sigma(z)| - e^{-md_\infty(z, p)} \right| < \xi. $

\end{enumerate}

\end{corollary}

We have the corollary below by following the same argument in the Lemma \ref{1to1-51}.
\begin{corollary} The Lipschitz map $\iota_\infty$ is one-to-one. In particular, 
$$(X_\infty, d_\infty, \mu_\infty) = (\hat X, d_\omega, \mu_\omega).$$

\end{corollary}

We have finally completed the proof of Theorem \ref{thm3gene} by combining the above results.

%%%%%%%%%%%%%%%%%%%%%%%%%%%%%%%%%%%%%%%%%%%%

\section{Singular K\"ahler metrics that are locally smoothable}

In this section, we establish the RCD structure for singular K\"ahler spaces that admit suitable algebraic smoothings. 

We first recall the well-known notions of algebraic smoothing defined as below (c.f. \cite{GGZ}). 

\begin{definition}\label{smoothing} Let $\cX \subset\mathbb{C}^N$ be an $(n+1)$-dimensional bounded normal Stein space equipped with a surjective holomorphic map
$$\pi: \cX \rightarrow \mathbb{D} \subset \mathbb{C}$$ satisfying following conditions. 

\begin{enumerate}

\item $\cX$ is $\mathbb{Q}$-Gorenstein and the relative canonical sheaf $K_{\cX/\mathbb{D}}$ is numerically trivial.

\smallskip

\item For any $t\in \mathbb{D}\setminus \{0\}$, $\cX_t = \pi^{-1}(t)$ is smooth.  

%\item $\cX_0$ has one and only one isolated log terminal singularity $\bfq$.  

\end{enumerate}
Then the central fiber $\cX_0$ is said to admit a 'smoothing'.

\end{definition}

\begin{definition} \label{qsmoothing} Let $(\bfX, \bfx)$ be a germ of an isolated log terminal singularity. $(\bfX, \bfx)$  is said to be $\mathbb{Q}$-smoothable if there exists a finite Galois quasi-etale covering $\Psi : \bfY\rightarrow X$ satisfying the following.

\begin{enumerate}

\item $\bfY$ is normal.

\smallskip

\item For any $\bfy \in \Psi^{-1}(\bfx)$, $(\bfY, \bfy)$ is smoothable as in Definition \ref{smoothing}. 

\end{enumerate}

\end{definition}

The following is the main result of this section. 

\begin{theorem} \label{locrcd7} Let $(\bfX, \bfx)$ be a germ of  $\mathbb{Q}$-smoothable isolated log terminal singularity with $\dim_{\mathbb{C}} \bfX =n$.   Suppose $\omega  $ is a singular K\"ahler metric on $(\bfX, \bfx)$ satisfying the following conditions.

\begin{enumerate}

\item $\omega \in C^\infty(\bfX \setminus \{\bfx\})$.

\smallskip

\item There exists $p>1$ such that $\frac{\omega^n}{\Omega_{\bfX}} \in L^p(\bfX, \Omega_{\bfX})$ for a smooth adapted volume measure $\Omega_{\bfX}$ on $\bfX$.

\smallskip

\item $\ric(\omega) \geq - \omega$ in the sense of currents.

\end{enumerate}
Then  $(\hat \bfX, d_\omega, \mu_\omega)$, the metric completion of $(\bfX\setminus \{\bfx\}, \omega, \omega^n)$ near $\bfx$,  satisfies the following.
\begin{enumerate}

\item $\hat \bfX = \bfX $ with $\cR(\hat \bfX) =\bfX \setminus \{\bfx\}. $

\smallskip

\item There exists $\epsilon=\epsilon(n)>0$ such that 
$$
\nu_{\hat \bfX} (\bfx) < 1- \epsilon. 
$$
\end{enumerate}
\end{theorem}

\subsection{Dirichlet problems}

In this section, we will prove various analytic and geometric estimates for singular K\"ahler spaces that admit a smoothing as in Definition \ref{smoothing}.

Let $(\bfX, \bfx)$ be a germ of isolated log terminal singularity with $\dim_{\mathbb{C}}\bfX = n$. Suppose $\bfX$ admits a smoothing as in Definition \ref{smoothing} with 
$$\pi: \cX \rightarrow \mathbb{D}, ~ \bfX = \cX_0. $$  Since $\cX_0$ has at worst log terminal singularities, we can assume that the total space $\cX$ as at worst canonical singularities and in this case the only singularity of $\cX$ is $\bfx$.

Since $\cX$ is normal, we can choose a smooth strict PSH function $\rho$ on $\cX$ with $\bfx\in \cU = \{ \rho < 1\}$. Obviously, $\cU$ is a strongly pseudoconvex domain and we can assume that $\partial \cU$ is smooth by perturbing $\rho$ if neccesary. We let $\cU_t = \cU\cap \cX_t$ and  assume that for each $t\in \mathbb{D}$, $\partial \cU_t$ is smooth. Again, $\cU_t$ is a strongly 
pseudoconvex domain for each $t\in \mathbb{D}$.  For convenience, we let $\mathbb{D}^* = \mathbb{D}\setminus \{0\}$. 

Let $\sigma$ be a local generator of $mK_\cX$ for some fixed $m\in \mathbb{Z}$. Then 
$$\Omega_\cX = \left( \sigma \overline{\sigma}\right)^{1/m}$$
 is a smooth adapted volume measure on $\cX$. The relative volume measure $\Omega_t$ can be defined by
 $$\sqrt{-1} \Omega_t \wedge dt \wedge d\bar t = \Omega_\cX. $$
 Both $\Omega_\cX$ and $\Omega_t$ have vanishing curvature, i.e., 
 $$\ric(\Omega_\cX) = - \ddbar \log \Omega_\cX = 0, ~ \ric(\Omega_t) = - \ddbar \log \Omega_t=0 $$
 for each $t\in \mathbb{D}$.

For fixed $f\in C^\infty(\overline{\cU})$ and $h\in C^\infty(\partial \cU)$, we let  
$$f_t = f|_{\overline{\cU_t}}, ~ h_t= h|_{\partial \cU_t}. $$
and consider the following family of Dirichlet problems on $\cU_t$ for $t\in \mathbb{D}$. 
\begin{equation} \label{dpA1}
\left\{
\begin{array}{l}
(\ddbar u_t )^n = e^{ \lambda u_t - f_t } \Omega_t  , ~ {\rm in} ~ \cU_t,  \\
\\
u_t =h_t, ~{\rm on} ~\partial \cU_t,
\end{array} \right.
\end{equation}
where $\lambda = 0$ or $1$. 
Since $\cU_t$ is strongly pseudoconvex, the following lemma is an immediately consequence of the classical works of \cite{CKNS}.
\begin{lemma} For every $t\in \mathbb{D}^*$, there exists a unique  solution $u_t \in C^\infty(\cU_t)\cap C^2(\overline{\cU_t})$ solving the Dirichlet problem (\ref{dpA1}).

\end{lemma}

Let $\omega_t = \ddbar u_t$ for each $t\in \mathbb{D}^*$. Since the right hand side of equation (\ref{dpA1}) is strictly positive in $\cU_t$ for each $t\in \mathbb{D}^*$, $\omega_t$ must be a K\"ahler metric on $\cU_t$.  

We now fix a smooth K\"ahler metric $\theta$ on $\cX$. There exists $\phi\in C^\infty(\cX)$ such that 
\begin{equation}\label{thetapot}
\theta=\ddbar \phi.
\end{equation}
We can also let $\theta_t = \theta|_{\cX_t}$ and $\phi_t=\phi|_{\cX_t}$ for each $t\in \mathbb{D}$. Obviously, $\theta_t$ is also a smooth K\"ahler metric on $\cX_t$ by definition.

The following lemma is proved in \cite{GGZ}, which can be viewed as the family version of the $L^\infty$-estimates for the Dirichlet problem of complex Monge-Amp\`ere equations. The key ingredient is the uniform control of Tian's $\alpha$-invariant for the fibers. 

\begin{lemma} Suppose there exist $p>1$ and $K>0$ such that for all $t\in \mathbb{D}^*$
\begin{equation}\label{tlp}
\left\|e^{-f_t}\frac{\Omega_t}{\theta_t^n} \right\|_{L^p(\cU_t, \Omega_t)} \leq K, ~ ||h_t||_{L^\infty(\partial X_t)}\leq K.
\end{equation}
Then there exists $C>0$ such that for all $t\in \mathbb{D}^*$, 
\begin{equation}
||u_t||_{L^\infty(\cU_t)} \leq C. 
\end{equation}

\end{lemma}

\begin{lemma}  In addition to (\ref{tlp}), we assume there exist $A, B>0$ such that for all $t\in \mathbb{D}$,
\begin{equation}\label{tschw1}
A\theta_t + \ddbar f_t  \geq 0
\end{equation}
in $\cU_t$ and  
\begin{equation}\label{tbound}
||h_t||_{C^2(\partial X_t)}\leq B.
\end{equation}
Then there exists $c>0$  such that for all $t\in \mathbb{D}^*$, 
\begin{equation}\label{tschw77}
\ddbar u_t \geq c\theta_t. 
\end{equation}
Furthermore, for each $t\in \mathbb{D}^*$, we have 
\begin{equation}\label{locricbd}
\ric(\omega_t) \geq - \left(1 + c^{-1}A \right) \omega_t. 
\end{equation}

\end{lemma}

\begin{proof}   The uniform boundary estimates for $u_t$ is obtained in \cite{GGZ}. More precisely, there exists $C>0$ such that for all $t\in \mathbb{D}^*$, $\|\Delta_{\theta_t} u_t \|_{\partial \cU_t} \leq C$. We can derive the interior second order estimates by applying the maximum principle to
$$H= \log tr_{\omega_t}(\theta_t) - B \phi_t $$ in $\cU_t$ with 
$$\Delta_{\omega_t} H \geq (B- C) tr_{\omega_t}(\theta_t)$$ for some uniform $C>0$. We then immediately achieve (\ref{tschw77}) by choosing a sufficiently large $B>0$. (\ref{locricbd}) follows from (\ref{tschw77})
$$\ric(\omega_t) = -\lambda \omega_t + \ddbar f_t \geq - \lambda \omega_t - A \theta_t \geq - \left( \lambda + c^{-1}A\right) \omega_t. $$
\end{proof}

%For any $\cK\subset\subset \cX \setminus \{ \bfx \}$, the lower bound estimate \ref{tschw77} for $\ddbar u_t$  immediately gives uniform second order estimates for $u_t$ on $\cK\cap \cU_t$ for all $t\in \mathbb{D}^*$ by combining with the equation (\ref{dpA1}). The standard Schauder estimates and elliptic theory would give higher order estimates for $u_t$ on $\cK\cap \cU_t$. 

\begin{corollary} Suppose (\ref{tlp}), (\ref{tschw1}), (\ref{tbound}) hold. Then for any $\cK \subset\subset \cU \setminus \{\bfx\}$ and $m>0$, there exists $C=C_{\cK, m}>0$ such that for all $t\in \mathbb{D}^*$, 
\begin{equation}
\| u_t \|_{C^m(\cK\cap \cU_t)} \leq C.  
\end{equation}
In particular, there exists a unique solution $u_0\in L^\infty(\overline{\cU_0})\cap C^\infty(\cU_0\setminus \{ \bfx \})$ to the Dirichlet problem (\ref{dpA1}) for $t=0$.

\end{corollary}

\begin{proof}  The uniform lower bound (\ref{tschw77}) for $\ddbar u_t$ combined with the original equation (\ref{dpA1}) would give the uniform upper bound for $\ddbar u_t$ in $\cK\cap \cU_t$. Then the higher order estimates for $u_t$ follow from standard Schauder estimates and the elliptic theory of linear PDEs. Now that $u_t$ is uniformly bounded in $L^\infty(\overline{\cU_t})$ and converges smoothly away from $\bfx$, the limit $u\in L^\infty(\overline{\cU_0}) \cap C^\infty( \cU_0)$ must solve the Dirchlet problem (\ref{dpA1}) on $\cU_0$ with $t=0$. The uniqueness directly follows from the comparison principle. 
\end{proof}

We let $\omega_t = \ddbar u_t$ and let $g_t$ be the K\"ahler metric associated to $\omega_t$ for all $t\in \mathbb{D}^*$.

\begin{lemma} Suppose (\ref{tlp}), (\ref{tschw1}), (\ref{tbound}) hold. Then there exists $C>0$ such that for each $t\in \mathbb{D}^*$, 
\begin{equation}\label{locdiam}
\Diam(\cU_t, g_t) \leq C. 
\end{equation}

\end{lemma}

\begin{proof} We follow the idea in \cite{FGS} (c.f. Lemma 2.4 \cite{FGS}). We will choose  a sufficiently small open neighborhood $V$ of $\bfx$ and let $V_t =V\cap \cU_t$. Since we have uniform second order estimates for $u_t$ away from $\bfx$, we can assume that there exist $V\subset\subset W \subset\subset \cU\setminus \{\bfx\}$ such that there exist  $\Lambda>0$ and $\delta>0$  such that for all $t\in \mathbb{D}^*$ and $x, y\in \cU_t\setminus W_t$,
\begin{equation}\label{distass}
d_{g_t}(x, y) \leq \Lambda, ~ d_{g_t} (\partial W_t, \partial \cU_t) \geq \delta,
\end{equation}
where $W_t = W\cap \cU_t$. 

We now fix  a sufficiently small $\epsilon=\epsilon(p)>0$ so that $p- \epsilon>1$, where $p$ is given in the assumption (\ref{tlp}). Without loss of generality,  we may assume  that 
$$\Diam(V_t, g_t) = 100D$$ for some very large $D\geq 100 + 100\Lambda$ and $t\in \mathbb{D}^*$. We now fix such $t\in \mathbb{D}^*$.

We can assume that there exists a point $p \in V_t$ such that $d_{g_t}(p, \partial V_t) = 10D$. Therefore  there exists a normal minimal geodesic  $\gamma:[0,D]\to W_t$ with respect to the metric $g_t$ such that its the tubular neighborhood $\gamma$ with radius $10$, i.e., 
$$\cV_\gamma= \{ x\in \cU_t: d_{g_t}(x, \gamma) < 10\}, $$ 
 entirely lies in $W_t$ due to the estimate (\ref{distass}).

We can choose the points $\{x_i = \gamma(6i)\}_{i=0}^{[D/6]}$ on $\gamma$  such that the geodesic balls $\{B_{g_t}(x_i,3)\}_{i=0}^{[D/6]} \subset W_t $ are  disjoint. There exists a uniform $\cA>0$ such that  
\begin{eqnarray*}
&& \sum_{i=0}^{[D/6]} \Big( \vol_{\theta_t^n}\xk{B_{g_t}(x_i,3)   }  + \vol_{e^{\lambda u_t-f_t}\Omega_t}\xk{B_{gt}(x_i, 3)}   \Big) \\
&\leq& \int_{\cU_t} \left( \theta_t^n + e^{\lambda u_t - f_t} \Omega_t \right)  \leq \left( 1 + e^{\lambda\|u_t\|_{L^\infty(\overline{\cU_t})}}    \left\|\frac{e^{-f_t}\Omega_t}{\theta_t^n} \right\|_{L^p(\overline{\cU_t}, \theta_t^n)}\right) \int_{\cU_t} \theta_t^n\\
&=&  \cA . 
\end{eqnarray*}
 Hence there must exist a geodesic ball $B_{g_t}(x_i, 3)$ such that \textcolor{black}{ $$\vol_{\theta_t^n}\xk{B_{g_t}(x_i,3)   }  + \vol_{e^{\lambda u_t - f_t} \Omega_t }\xk{B_{g_t}(x_i, 3)} \le 12 \cA D^{-1}.$$ } We fix such $x_i$ and construct a cut-off function $\eta(x) = \rho(r(x))  \geq 0$ with 
$$r(x) = d_{g_t} (x,x_i)$$ such that 
$$\eta=1~\textnormal{on}~ B_{g_t}(x_i, 1),  ~ \eta=0 ~\textnormal{outside}~B_{g_t} (x_i, 2), $$
and
$$ \rho\in [0,1] ,\quad \rho^{-1} (\rho')^2\le 10, \quad |\rho''| \leq 10. $$
 Define a piecewise linear continuous function $\tilde F: \mathbb R \to \mathbb R$ such that $F(s) = D^{\frac{\epsilon}{p(p-\epsilon)}}$ when $t\in [0,2]$, $\tilde F(s) \equiv a$ when $s\ge 3$ and $\tilde F(s) $ is linear when $t\in [2,3]$, where $a>0$ is a constant to be determined.  Denote $F(x) = \tilde F( r(x)   )$, then  $F\equiv a$ outside $B_{g_t} (x_i, 3)$, $F= D^{\frac{\epsilon}{p(p-\epsilon)}}$ on $B_{g_t}(x_i, 2)$. We can mollify $F$ so that it is smooth  without changing its value in $B_{g_t}(x_i, 2)$.
 %
%We choose the constant $a>0$  so that $\int_{\cU_t} F e^{-f_t}\Omega_t \leq  \cA$. 
%We observe that
%\begin{align*}V& = a \vol_\Omega( X\backslash B_{g_t}(x_i,3)  ) + \int_{B_{g_t}(x_i,3)} F \Omega\ge V(1- 12 D^{-1}) a\ge \frac{V}{2} a
%\end{align*}
%so $0<a\le 2$. 
%
$$\int_{\cU_t} \left(\frac{F e^{\lambda u_t - f_t} \Omega_t}{\theta_t^n} \right)^{p-\epsilon} \theta_t^n \leq \left( \int_{\cU_t} F^{\frac{p(p-\epsilon)}{\epsilon} } \theta_t^n\right)^{\frac{\epsilon}{p}}\left( \int_{\cU_t} \left(\frac{\Omega_t}{\theta_t^n} \right)^p \theta_t^n \right)^{ \frac{p-\epsilon}{p}} \leq \cA.$$

We now consider the equation
\begin{equation} \label{dpA2}
\left\{
\begin{array}{l}
(\ddbar w )^n =   F e^{\lambda u_t - f_t} \Omega_t  , ~ {\rm in} ~ \cU_t,  \\
\\
w =h_t, ~{\rm on} ~\partial \cU_t,
\end{array} \right.
\end{equation}
Then by the $L^\infty$-estimate (\ref{tlp}) and the uniform $L^\infty$-bound for $u_t$, we have $\| w \|_{L^\infty} \leq C$ for a uniform $C>0$. Furthermore $w$ is smooth on $\cU_t$ and 
$$\hat g=   \ddbar w$$
is a K\"ahler metric in $\cU_t$. On $B_{g_t}(x_i, 2)$, $F$ is constant and so 
$$\ric(\hat g) = \ric(g_t). $$
In particular, we have
$$ \Delta_{g_t} \log \frac{\hat\omega^n} {\omega_t^n} = 0  $$
 on $B_{g_t}(x_i, 2)$, where $\Delta_{g_t} = \Delta_{\omega_t}$. Let 
$$H = \eta \left( \log \frac{\hat\omega^n} {\omega_t^n}  + w  \right). $$  
On $B_g(x_i, 1), $ we have $\eta=1$ and 
$$ \Delta_{g_t}  H = \Delta_{g_t} w =   \tr_{g_t}(\hat g) \geq   n \left( \frac{\hat\omega^n} {\omega_t^n} \right)^{1/n} .$$
In general,  on the support of $\eta$, we have 
\begin{eqnarray*}
\Delta_{g_t} H &\geq& n \eta  \left(\frac{\hat \omega^n} {\omega^n} \right)^{1/n}   + 2 \eta^{-1} Re\left(  \nabla H \cdot  \nabla \eta \right) - 2 \frac{ H  |\nabla \eta|^2}{\eta^2} + \eta^{-1}H \Delta_{g_t} \eta \\
&\geq &  \eta^{-1}\bk{C \eta^2 e^{H/(n\eta)} + 2 Re\big(  \nabla H \cdot  \nabla \eta \big) - 2 \frac{ H  |\nabla \eta|_{g_t}^2}{\eta} + H \Delta_{g_t} \eta }.
%
%& \geq & C \eta \left( e^{H/\eta} - C^2( H/\eta)  - 1 \right)+ 2 \eta^{-1} Re\left(  \nabla H \cdot  \nabla \eta \right).
%
%
\end{eqnarray*}
We may assume $\sup_{\cU_t} H>0$, otherwise we already have upper bound of $H$. The maximum of $H$ must lie at $B_{g_t}(x_i,2)$ and at this point $$\Delta_{g_t} H\le 0 ,\quad \abs{\nabla H}_{g_t} = 0.$$ By the Laplacian comparison with the lower bound of $\ric(g_t)$ from (\ref{locricbd}), there exists $C>0$ such that 
$$\Delta_{g_t} \eta = \rho' \Delta r + \rho'' \ge -C,\quad \frac{\abs{\nabla \eta}_{g_t}}{\eta} = \frac{(\rho')^2}{\rho}\le C.$$ 
Hence at the maximum of $H$, we have
$$0\ge C \eta^2 e^{H/(n\eta)} - C H \ge C H^2 - C H , $$ 
and  
$$\sup_{\cU_t} H\le C $$
for some uniform $C>0$. 
 In particular, on the ball $B_{g_t}(x_i,1)$ where $\eta\equiv 1$, it follows that $\frac{\hat \omega^n}{\omega_t^n}\le \cB$ for some uniform $\cB>0$ since $\| w\|_{L^\infty(\cU_t)}$ is bounded by a uniform constant. From the definition of $\hat\omega$ and $\omega_t$, we have
$$\cB \ge  \frac{\hat \omega^n}{\omega_t^n} = F=D^{\frac{\epsilon}{p (p-\epsilon)}}$$
on $B_{g_t}(x_i, 1)$. Immediately we have a uniform upper bound for 
$$D \leq \cB^{\frac{p(p-\epsilon)}{\epsilon}}$$
and so we have obtained a uniform upper bound for $\Diam(V_t, g_t)$ for all $t\in \mathbb{D}^*$. 
\end{proof}

We now let $u=u_0$ and $\omega = \ddbar u_0$ which are both smooth in $\cU_0\setminus \{\bfx\}$. Naturally, we can let $(\hat \cU_0, d_\omega, \mu_\omega)$ be the metric completion of $(\cU_0\setminus \{\bfx\}, \omega, \omega^n)$ around $\bfx$ (but not near $\partial \cU_0$). 

We will now state and prove the main result in this subsection.

\begin{lemma} \label{rcdloc75} $(\cU_t, \omega_t, \omega_t^n)$ converge in Gromov-Hausdorff topology to an open non-collapsed $\RCD(1+c^{-1}A, 2n)$ space $(\cW, d_\cW, \mu_\cW)$ satisfying the following conditions.

\begin{enumerate}

\item $\cW = \hat \cU_0=\cU_0$ with $\cR(\cW) = \cU_0 \setminus \{\bfx\}. $

\smallskip

\item There exists $\epsilon=\epsilon(n)>0$ such that 
\begin{equation}\label{locvolden2}
\nu_\cW (\bfx) < 1- \epsilon. 
\end{equation}

\end{enumerate}

\end{lemma}

\begin{proof}  As before, we fix an open subset $V \subset\subset \cU$ with  $\bfx \in V$  and let $V_t = V\cap \cU_t$. By the Ricci lower bound (\ref{locricbd}), diameter bound (\ref{locdiam}) and the smooth convergence of $\omega_t$ to $\omega$ away from $\bfx$, there exists $C>0$ such that for  any $t\in \mathbb{D}^*$ and any   $p\in V_t$, 
$$\vol(B_{g_t}(p, 1), g_t) \geq C^{-1}. $$ 

Applying the Ricci lower bound  and diameter bound together, we can infer that $(\cU_t, \omega_t, \omega_t^n)$ converge in Gromov-Hausdorff topology to an open and bounded non-collapsed $\RCD(1+c^{-1}A, 2n)$ space $(\cW, d_\cW, \mu_\cW)$ by the Cheeger-Colding theory. The convergence is smooth on $\cU_0\setminus \{\bfx\}$ and so $\cU_0\setminus \{\bfx\} \subset \cR(\cW)$. 
By the estimate (\ref{tschw77}), the identity map from $\cU_0\setminus \{\bfx\} \subset \cW$ to itself in $\cU_0$ extends to a surjective Lipschitz map
\begin{equation}\label{lociota}
\iota: \cW \rightarrow \cU_0.
\end{equation}

Each $\omega_t$ can be viewed as K\"ahler metric in the class of a trivial line bundle, therefore we can apply the partial $C^0$-estimates in \cite{DS1, LS1} based on the geometric $L^2$-estimates in Lemma \ref{L242} and Lemma \ref{L241} as argued in Lemma \ref{1to1-51}.   By the same argument as in Lemma \ref{dengap},  we can prove that if $p\in \cW$ is a metric regular point, then $\iota(p)\neq \bf{x}$. Immediately, we have $\cR(\cW) = \cU_0 \setminus \{\bfx\}$.  Then we show that $\iota$ is injective. In particular, for any two distinct points $p, q \in \cW$, there exists a holomorphic function $\sigma$ on an open neighborhood of $\iota(p)$ and $\iota(q)$ such that $\iota^*\sigma$ extends to a holomorphic function in an open neighborhood of $p$ and $q$ with $\iota^*\sigma(p)\neq \iota^*\sigma(q)$. This implies the injectivity of $\iota$ and so $$\cW = \cU_0. $$

The volume density estimate (\ref{locvolden2}) follows from the same argument of Lemma \ref{dengap}. 
\end{proof}

%%%%%%%%%%%%%%%%%%%%%%%%%%%%%%%%%%%%%%%%%%%%%%%

%%%%%%%%%%%%%%%%%%%%%%%%%%%%%%%%%%%%%%%%%

\subsection{Proof of Theorem \ref{locrcd7}}  We will break the proof into two steps.

\medskip

\noindent{\it Step 1.}  We first assume that $(\bfX, \bfx)$ admits a smoothing in the sense of Definition \ref{smoothing}.  Let $\pi: \cX \rightarrow \mathbb{D}$ be the total space for the smoothing of $(\bfX, \bfx)$ with $\bfX = \cX_0= \pi^{-1}(0)$. We will follow the same notations as in the discussion in Section 7.1 such as $\cU$, $\cU_t$, $\Omega_\cX$, $\Omega_t$. 
Let $\alpha = \ric(\omega) + \omega$. Since $\alpha\geq 0$, there exists $\psi\in \PSH(\bfX)$ with $\alpha = \ddbar \psi$. Since $\cX$ and $\cX_0$ are algebraic and can be embedded in projective spaces, $\psi$ can be extended to $\tilde \psi \in \PSH(\cX)$ by the extension theorem of \cite{CG}  and we let $\tilde \psi_t = \tilde \psi|_{\cX_t}$. 

\medskip

\noindent{\it(1.1).} We will first assume that $\tilde \psi$ is smooth in $\cX$. Then $u$ (modulo a constant) satisfies the following complex Monge-Amp\`ere equation on $\cU_0$ 
$$(\ddbar u)^n = e^{u - \tilde \psi_0} \Omega_0. $$
We now extend $u|_{\partial \cU_0}$ to a smooth function $\tilde h$ in a neighborhood of $\partial \cU$ since $u$ is smooth away from $\bfx$. Let $\tilde h_t = \tilde h|_{\partial \cU_t}$. We can consider the following family of complex Monge-Amp\`ere equations
\begin{equation} \label{dpA10}
\left\{
\begin{array}{l}
(\ddbar u_t )^n = e^{ \lambda u_t - \tilde\psi_t } \Omega_t  , ~ {\rm in} ~ \cU_t,  \\
\\
u_t =\tilde h_t, ~{\rm on} ~\partial \cU_t,
\end{array} \right.
\end{equation}
Then the theorem can be proved by directly applying Lemma \ref{rcdloc75} and the other results in Section 7.1. 

\medskip

\noindent{\it (1.2).}  If $\tilde \psi$ is not smooth, there exists a sequence of smooth PSH functions $\tilde \psi_j$ that decreasingly converge to $\tilde \psi$ by \cite{CG}. Let $\tilde \psi_{j, t} = \tilde \psi_j |_{\cU_t}$. Since $e^{-\tilde \psi_{j, t} }  \leq e^{-\tilde \psi_t}$, $e^{-\tilde \psi_{j, t} } \frac{\Omega_t}{\theta_t^n}$ are uniformly bounded in $L^p(\cU_t, \theta_t^n)$ for some $p>1$ and for all $t\in \mathbb{D}^*$ and all $j$. The theorem can be proved by applying the result in (1.1) and  taking a diagonal sequence for $t\rightarrow 0$ and $j\rightarrow \infty$ since the bounds in (1.1) are uniform and independent of $j$ and $t$.

\medskip

\noindent{\it Step 2.} Suppose $(\bfX, \bfx)$ is $\mathbb{Q}$-smoothable. By shrinking $\bfX$ if necessary, there exists a Galois quasi-etale covering $\Psi:\tilde \bfX  \rightarrow \bfX$ and the pre-image $\tilde \bfx = \Psi^{-1}(\bfx)$ is a single point as the only fixed point of the Galois group action. By choosing a sufficiently small strongly pseudoconvex domain $U$ as an open neighborhood of $\bfx$, we can repeat Step 1 for $\tilde \omega = \Psi^*\omega$ on $\tilde U = \Psi^{-1} (U)$. The theorem still holds after taking the quotient by a finite isometric group action for the metric completion of $(\tilde U\setminus \{\tilde \bfx\}, \tilde \omega)$ whose metric singular set is the isolated point $\tilde \bfx$.  

\medskip

Combining the above, we have completed the proof of Theorem \ref{locrcd7}.

%%%%%%%%%%%%%%%%%%%%%%%%%%%%%%%%%%%%%%%%%

%%%%%%%%%%
%%%%%%%%%%%%%%%%%%%%%%%%%%%%%%%%%%%%%%%%%%%%

\section{Proof of Theorem \ref{mainthm2}}

In this section, we consider a family of normal K\"ahler varieties which is slightly more general than those satisfying (\ref{spres}) in Section 5 and Section 6.  

Let $X$ be an $n$-dimensional projective variety of log terminal singularities. Throughout this section, we will assume there exists a partial resolution of singularities 
$$\pi: Y \rightarrow X$$ 
for $X$ satisfying the following conditions.

\begin{enumerate}

\item[(B.1)] $Y$ has only isolated $\mathbb{Q}$-smoothable singularities.

\medskip

\item[(B.2)]  $K_Y = \pi^*K_X - \sum_j a_j E_j, ~a_j \in [0, 1)$, 
where $E_j$ is a prime divisor. Such assumptions imply that $X$ has log terminal singularities. 

\end{enumerate}

There always exists a partial resolution satisfying condition (B.2) because such a resolution corresponds to a terminalization of $X$ which always exists due to the work of \cite{R80}. The condition (B.1) is also satisfied if in addition, $\dim_{\mathbb{C}} X= 3$ due to the classification of terminal singularities in dimension $3$ by \cite{R83}. 

%%%%%%%%%%%%%%%%%%%%%%%%%%%%%%%%%%%%%%%%

\medskip

\subsection{A generalization of Theorem \ref{thm3spe}}

The following is the main result of this section.

\begin{theorem} \label{thm4spe} Suppose $X$ is an $n$-dimensional projective variety of log terminal singularities satisfying  (B.1) and (B.2). Let $\omega\in \nu(X, \theta_X, n, A, p, K)$ be  twisted K\"ahler-Einstein metric $\omega$  on $X$ satisfying
\begin{equation}\label{tkespe88}
\ric(\omega) = - \omega + \alpha
\end{equation}
for a smooth non-negative closed $(1,1)$-form $\alpha$. Then the metric measure space $(\hat X, d_\omega, \mu_\omega)$ induced by $(X^\circ, \omega)$ is a compact non-collapsed $\RCD(-1, 2n)$ space satisfying the following properties. 
\begin{enumerate}

\item  $(\hat X, d_\omega)$ is homeomorphic to the original variety $X$.

\smallskip

\item   $\mathcal{R}(\hat X)= X^\circ$.

\smallskip

\item  There exists $c>0$ such that $$\omega \geq c \theta_X. $$
%
%\medskip

\item There exists $\epsilon=\epsilon(n)>0$ such that for any $p \in \cS(\hat X)$, 
$$ 
\nu_{\hat X} (p)< 1 -\epsilon.
$$

\end{enumerate}

\end{theorem}

%We can immediately apply Theorem \ref{thm4spe} to dimension $3$. 
%
%\begin{corollary}
%\label{thm4spe3d} Suppose $X$ is a 3-dimensional projective variety of log terminal singularities equipped with a smooth K\"ahler metric $\theta_X$. Let $\omega\in \nu(X, \theta_X, n, A, p, K)\cap C^\infty(X^\circ)$ be  a twisted K\"ahler-Einstein metric $\omega$  on $X$ satisfying
%
%\begin{equation}\label{tkespe89}
%
%\ric(\omega) = - \omega + \alpha
%
%\end{equation}
%
%for a smooth non-negative closed $(1,1)$-form $\alpha$. Then the metric measure space $(\hat X, d_\omega, \mu_\omega)$ induced by $(X^\circ, \omega)$ is a non-collapsed RCD space satisfying the following properties. 
%
%\begin{enumerate}

%\item  $(\hat X, d_\omega)$ is homeomorphic to the original variety $X$.

%\smallskip

%\item  $\mathcal{R}(\hat X)= X^\circ$.

%$\mathcal{S}(\hat X) = \mathcal{S}(X)$, where $\mathcal{S}(\hat X) $ is the singular set for the metric space $(\hat X, d, \omega^n)$ and $\mathcal{S}(X)$ is the singular set for the K\"ahler space $X$. 

%\smallskip

%\item  There exists $c>0$ such that $$\omega \geq c \theta_X. $$
%
%\medskip

%\item There exists $\epsilon>0$ such that for any $p \in \cS(\hat X)$, 
%
%$$
%
%\nu_{\hat X} (p)< -\epsilon.
%
%$$
%

%\end{enumerate}

%\end{corollary}

We will begin our proof for Theorem \ref{thm4spe}. We will first follow the set-up in Section 5.1 and will use the same notations as in Section 5.1. The only difference is that $Y$ is no longer necessarily smooth, instead it has finiately many isolated $\mathbb{Q}$-smtoothable singularities. We let $$\mathbf{Q} = \{ \mathbf{q}_1, ..., \mathbf{q}_N\} $$
 be the set of such isolated singularities and let 
 $$Y^\circ = Y \setminus \mathbf{Q}.$$
 
As in Section 5.1, we have the following equation for $\omega=\omega_X + \ddbar \varphi$
\begin{equation}\label{ma88}
(\omega_X + \ddbar \varphi)^n = e^{\varphi } \Omega_X. 
\end{equation}
as compared to (\ref{ma1}) with
$$\alpha = \ric(\Omega_X) + \omega_X.$$

Similarly we consider consider the following families of approximations for equation (\ref{ma88}). 
\begin{equation}\label{ma83}
(\omega_X + \delta_k \theta_Y + \ddbar \varphi_k )^n = e^{ \varphi_k} \prod_{i=1}^I |\sigma_i|_{h_i}  ^{-2a_i}\Omega_Y, 
\end{equation}
and  
\begin{equation}\label{ma84}
(\omega_X + \delta_k  \theta_Y + \ddbar \varphi_{k,  j} )^n = e^{ \varphi_{k,  j}} \prod_{i=1}^I \left( |\sigma_i|^2_{h_i} + \epsilon_j \right)^{-a_i}\Omega_Y. 
\end{equation}
with
$$\delta_k, ~\epsilon_j \in (0,1), ~ \lim_{k\rightarrow \infty}\delta_k = \lim_{j\rightarrow \infty} \epsilon_j = 0 $$
and
$$
\omega_k = \omega_X + \delta_k \theta_Y + \ddbar \varphi_k , ~\omega_{k,j} = \omega_X + \delta_k \theta_Y + \ddbar \varphi_{k,j}. 
$$

By almost identical argument as in Section 5.1 and 5.2, we have the following lemma. 

\begin{lemma}  There exists $C>0$ such that for all $j, k>0$,  
\begin{enumerate}

\item $\|\varphi_{k, j} \|_{L^\infty(Y)} + \|\varphi_k\|_{L^\infty(Y)}\leq C, $

\smallskip

\item $\omega_k \geq c \theta_X,  $

\smallskip
\item $\Diam \overline{(Y^\circ, \omega_{k, j})} \leq C$ and $ \Diam \overline{(Y^\circ\setminus E', \omega_k) } \leq C, $
\smallskip
\item $  \ric(\omega_k) \geq - \omega_k$  on  $Y^\circ$, 

\smallskip

\item  $\ric(\omega_{k, j} ) \geq - Ce^{C\delta_k^{-1}} \omega_{k, j}$ on  $Y^\circ\setminus E'$.

\end{enumerate}

\end{lemma} 

We have the following lemma similar to Lemma \ref{highreg1} and Lemma \ref{highreg2}.
\begin{lemma}   For any $\cK \subset \subset Y^\circ \setminus E'$, $m>0$ and $k>0$, there exist $C=C_{\cK, m, k}>0$ such that for all $j>0$, 
$$\|\varphi_{k,  j} \|_{C^m(\cK)} \leq C.$$  
For any $\cK' \subset \subset Y^\circ \setminus E$ and $m>0$, there exist $C'=C'_{\cK, m}>0$ such that for all $j>0$, 
$$\|\varphi_{k, j}\|_{C^m(\cK)} \leq C' , ~\|\varphi_k \|_{C^m(\cK)} \leq C' .$$
\end{lemma}
In fact, $\varphi_{k, j}$ converges smoothly to $\varphi_k$ on $Y^\circ\setminus E'$ as $j\rightarrow \infty$ and $\varphi_k$ converge  smoothly to $\varphi$ on $Y^\circ\setminus E$ as $k\rightarrow \infty$.

Let $$(Y_{k, j}, d_{k, j}, \mu_{k, j})= \overline{ (Y^\circ, \omega_{k,j}, \omega_{k, j}^n)}.$$   The following lemma is the only new result. 
 
 \begin{lemma} There exists $C>0$ such that for all $k, j>0$, $(Y_{k, j}, d_{k, j}, \mu_{k, j})$ is a non-collapsed $RCD\left(-Ce^{\delta_k^{-1}}, 2n \right)$ space satisfying the following conditions.
 
 \begin{enumerate}
 
 \item $Y_{k, j}$ is homeomorphic to $Y$ for all $k, j>0$.
 
 \item $\cS(Y_{k, j}) = \mathbf{Q}$ and there exists $\epsilon = \epsilon(n)>0$ such that for any point $\mathbf{q} \in \mathbf{Q}$
 \begin{equation} \label{singvolq}
 \nu_{Y_{k, j}}(\mathbf{q}) < 1 - \epsilon. 
 \end{equation}

 \end{enumerate}

 \end{lemma}
 
\begin{proof} For each fixed pair $k, j>0$, $\varphi_{k, j}$ is smooth on $Y^\circ$. For each $\bfq\in \bfQ$, we can choose a sufficiently small strongly pseudoconvex domain $U_\bfq$ containing $\bfq$. By choosing $U_\bfq$ sufficiently small, we can assume that $\omega_X+ \delta_k \theta_Y = \ddbar \phi_k$ for some smooth $\phi_k$.  We can consider the following family of Dirichlet problems
\begin{equation} \label{dp88}
\left\{
\begin{array}{l}
(\ddbar u_{k, j} )^n = e^{ u_{k, j} - \phi_k} \prod_{i=1}^I \left( |\sigma_i|^2_{h_i} + \epsilon_j \right)^{-a_i}\Omega_Y  , ~ {\rm in} ~U_\bfq,  \\
\\
u_{k, j} =\phi_k+\varphi_{k, j}, ~{\rm on} ~\partial U_\bfq,
\end{array} \right.
\end{equation}%
Obviously, $u_{k, j}|_{\partial U_\bfq} \in C^\infty(\partial U_\bfq)$ and $u_{k, j}= \phi_k + \varphi_{k, j}$ is the unique solution of (\ref{dp88}) with 
$$\omega_{k, j} = \ddbar u_{k, j}, ~\ric(\omega_{k, j}) \geq  - Ce^{C\delta_k^{-1}}\omega_{k, j}$$ in $U_\bfq$. Since $\bfq$ is $\mathbb{Q}$-smoothable, we can apply Theorem \ref{locrcd7} and  $(U_\bfq, \omega_{k, j})$ is a local non-collapsed $\RCD(-Ce^{\delta_k^{-1}}, 2n)$ space. Furthermore, the volume density of $\bfq$ is uniformly bounded above away from $1$. 
The lemma follows since $\omega$ is smooth away from $\bfQ$.
\end{proof}

In fact, the RCD space $(Y_{k, j}, d_{k, j}, \mu_{k, j})$ is an almost smooth metric measure space associated with the open smooth subset $Y^\circ\setminus E'$. 
We will now follow similar arguments in Section 5.3 that will be applied to  $(Y_{k, j}, d_{k, j}, \mu_{k, j})$ for fixed $k>0$ and $j=1, 2, ...$   

\begin{lemma}  \label{ykspace81} For each fixed $k>0$, $(Y_{k, j} , d_{k, j}, \mu_{k, j})$ converge in Gromov-Hausdorff sense to a compact $\RCD(-1, 2n)$ space $( Y_k,  d_k, \mu_k)$ as $j \rightarrow \infty$ satisfying the following. 

\begin{enumerate} 
 
 \item $(Y_k, d_k) = \overline{(Y^\circ\setminus E', \omega_k)} = Y. $
 
 \smallskip
 
 \item There exists $\epsilon=\epsilon(n)>0$ such that for any $\bfq \in \bfQ$, $\nu_{Y_k}(p) < 1 - \epsilon. $
 
 \smallskip
 
 \item For any $p\in E'$, $\nu_{Y_k}(p) \leq 1- \min_{i\in I, a_i>0}  a_i $. 
 
 \end{enumerate}
 
 \end{lemma}

\begin{proof} The proof follows from an almost identical argument in Section 5.2. We will only highlight the parts that involve  $\bfQ$. For each fixed $k>0$, $\ric(\omega_{k, j})$ is uniformly bounded below by $-Ce^{C\delta_k^{-1}}$ and the diameter of $(Y_{k, j}, d_{k, j})$ is uniformly bounded above for all $j>0$. Therefore $(Y_{k, j}, d_{k, j}, \mu_{k,j})$ converge in Gromov-Hausdorff topology to an $\RCD(-Ce^{\delta_k^{-1}}, 2n)$ space $( Y_k, d_k, \mu_k)$, after possibly passing to a subsequence.

The smooth convergence of $\omega_{k, j}$ on $Y^\circ \setminus E'$ as $j\rightarrow \infty$ also shows that $(Y_k, d_k, \mu_k)$ is an almost smooth metric measure space associated with $Y^\circ \setminus E'$. Let $\bfQ'$ be the Gromov-Hausdorff limits of $\bfQ$. By the local second order estimates analogous to (\ref{metlb51}), the distance between any pair of points in $\bfQ'$ are uniformly bounded below away from $0$ for all $\omega_{k, j}$ with $j>0$. Therefore $|\bfQ'|=|\bfQ|$. On the other hand, $Y_k \setminus \bfQ'$ is homeomorphic and in fact biholomorphic to $Y\setminus \bfQ$ by the same partial $C^0$-estimate argument in Lemma \ref{1to1-51}. Therefore the metric completion of $(Y_k \setminus \bfQ', d_k)$ must coincide with $Y=Y^\circ \cup \bfQ$.

We now can verify that  $( Y_k, d_k, \mu_k)$ is an $\RCD(-1, 2n)$ space. By the smooth convergence of $\omega_{k, j}$ on $Y^\circ\setminus E'$, $(Y_k, d_k, \mu_k)$ is an almost smooth compact  metric measure space associated with the open smooth subset $Y^\circ\setminus E'$. By (\ref{riclb51}) and the fact that $(Y_k, d_k, \mu_k)$ is already a non-collapsed RCD space with every eigenfunction being Lipschitz, we can apply Lemma \ref{honda3} and conclude that $( Y_k, d_k, \mu_k)$ is an $\RCD(-1, 2n)$ space.

Finally, the volume density estimates follow by the same argument from Lemma \ref{lem38} and the estimate (\ref{singvolq}).	
\end{proof}

%%%%%%%%%%%%%%%%%%%%%%%%%%%%%%%%%%%%

\medskip

We now  follow the same argument in Section 5.3 by considering the sequence of the non-collapsed $\RCD(-1, 2n)$ spaces $(Y_k, d_k, \mu_k)$. By letting $k \rightarrow \infty$, we obtain a limiting RCD space $(Y_\infty, d_\infty, \mu_\infty)$ after passing to a subsequence of $k$ by the compactness theory of non-collapsed RCD spaces.

\begin{lemma}  \label{ykspace82} $(Y_k, d_k, \mu_k)$ converges in Gromov-Hausdorff sense to a compact $\RCD(-1, 2n)$ space $( Y_\infty,  d_\infty, \mu_\infty)$ as $k \rightarrow \infty$ satisfying the following. 

\begin{enumerate} 
 
 \item $(Y_\infty, d_\infty) = (\hat X, d_\omega) = X$.
 
 \smallskip
 
 \item $\cR(Y_\infty) = X^\circ$. 
 
 \smallskip
 
 \item There exists $\epsilon=\epsilon(n)>0$ such that for any $q\in \cS(Y_\infty)=\cS(X)$, $$\nu_{Y_\infty}(p) < 1 - \epsilon. $$

 \item $\cS(Y_\infty) = \cS_{2n-3}(Y_\infty)$.
 
 \end{enumerate}
 
 \end{lemma}

\begin{proof} For each point $\bfq\in \bfQ$, $\pi(\bfq)\in \cS(X)$ and  the volume density of $\bfq$ in $(Y_k, d_k, \mu_k)$ is uniformly less than $1$ for all $k>0$. Therefore we can follow the argument in Section 5.3 without changes and prove the lemma. 
\end{proof}

We have now completed the proof of Theorem \ref{thm4spe} by combining the above results.
%%%%%%%%%%%%%%%%%%%%%%%%%%%%%%%%%%%%%%%%%%%%
%%%%%%%%%%%%%%%%%%%%%%%%%%%%%%%%%%%%%%%%%%%%

\medskip

\subsection{A generalization of Theorem \ref{thm3gene}} 

We will prove an analogue of Theorem \ref{thm3gene} by generalizing Theorem \ref{thm4spe}. 

\begin{theorem} \label{thm4gene} Suppose $X$ is $n$-dimensional projective variety with log terminal singularities satisfying  (B.1) and (B.2). Let $\omega\in \nu(X, \theta_X, n, A, p, K)$ be a singular K\"ahler  metric $\omega$  on $X$ satisfying
\begin{equation}\label{skmbdl9}
\ric(\omega) + \omega \geq 0
\end{equation}
in the sense of currents. Then the metric measure space $(\hat X, d_\omega, \mu_\omega)$ induced by $(X^\circ, \omega)$ is a compact non-collapsed $\RCD(-1, 2n)$ space satisfying the following properties. 
\begin{enumerate}

\item  $(\hat X, d_\omega)$ is homeomorphic to the original variety $X$.

\smallskip

\item $\mathcal{R}(\hat X)= X^\circ$.
 
\smallskip

\item  There exists $c>0$ such that $$\omega \geq c \theta_X. $$
%
%\medskip

\item There exists $\epsilon=\epsilon(n)>0$ such that for any $p \in \cS(\hat X)$, 
$$
\nu_{\hat X} (p)<1 -\epsilon.
$$

\end{enumerate}

\end{theorem}

 The proof of Theorem \ref{thm4gene} follows the same arugment for Theorem \ref{thm3gene} by applying Theorem \ref{thm4spe}.

%\begin{corollary} \label{thm4gene3d} Suppose $X$ is a 3-dimensional projective variety of log terminals singularities equipped with a smooth K\"ahler metric $\theta_X$. Let $\omega\in \nu(X, \theta_X, n, A, p, K)\cap C^2(X^\circ) \cap H^{1,1}(X, \mathbb{Q})$ be a singular K\"ahler  metric $\omega$  on $X$ satisfying
%
%\begin{equation}\label{skmbdl9}
%
%\ric(\omega) + \omega \geq 0
%
%\end{equation}
%
%in the sense of currents. Then the metric measure space $(\hat X, d_\omega, \omega^3)$ induced by $(X^\circ, \omega)$ is a non-collapsed RCD space satisfying the following properties. 
%
%\begin{enumerate}

%\item  $(\hat X, d_\omega)$ is homeomorphic to the original variety $X$.

%\medskip

%\item The regular part of $\hat X$ coincides with the regular part of $X$, i.e., $\mathcal{R}(\hat X)= X^\circ$.

%\medskip

%\item  There exists $c>0$ such that $$\omega \geq c \theta_X. $$
%
%\medskip

%\item There exists $\epsilon=\epsilon(n)>0$ such that for any $p \in \cS(\hat X)$, 
%
%$
%
%\nu_{\hat X} (p)< -\epsilon.$
%

%

%\end{enumerate}

%\end{corollary}

%%%%%%%%%%%%%%%%%%%%%%%%%%%%%%%%%%%%%%%%%%%

\medskip

\subsection{Proof of Theorem \ref{mainthm2}} %

We are now ready to prove Theorem \ref{mainthm2}. We first recall the theorem of Reid \cite{R80} (c.f. Corollary 5.38 \cite{KM}) for terminal singularities in complex dimension $3$.

\begin{lemma} \label{reid1} Let $(\bfX, \bfx)$ be the germ of a Gorenstein terminal singularity with $\dim_\mathbb{C} X = 3$.  Then $(\bfX, \bfx)$ is a compound du Val singularity. In particular, $(\bfX, \bfx)$ is smoothable in the sense of Definition \ref{smoothing}. 
\end{lemma}

By taking a quasi-etale covering, Lemma \ref{reid1} can be extended to general three dimensional terminal singularities. 

\begin{corollary} \label{reid2} Let $(\bfX, \bfx)$ be the germ of a $\mathbb{Q}$-Gorenstein terminal singularity with $\dim_{\mathbb{C}}X=3$. Then $(\bfX, \bfx)$ is $\mathbb{Q}$-smoothable in the sense of Definition \ref{qsmoothing}. 

\end{corollary}

Suppose $X$ is a three dimensional projective variety with log terminal singularities. There exists a terminalization  \cite{R83} (c.f. Theorem 6.23 \cite{KM})
$$\pi: Y \rightarrow X $$
such that $Y$ has only terminal singularities and  condition (B.2) holds.  Since $\dim_{\mathbb{C}} Y=3$, there are only finitely many isolated terminal singularities on $Y$ and they are all $\mathbb{Q}$-smoothable by Corollary \ref{reid2}. Therefore condition (B.1) also holds.  Theorem \ref{mainthm2} is now an immediate corollary of Theorem \ref{thm4gene}.

If $\omega$ has bounded Ricci curvature,  
we have the following dimension estimate for the singular set $X_\infty$ similar to Lemma \ref{lemcod4}, 

\begin{lemma} \label{lemcod42} Suppose $\ric(\omega)$ is further bounded above. Then $\cS(X_\infty) = \cS_{2n-4}$ and $$\dim_{\mathcal{H}} \cS(X_\infty) \leq 2n-4. $$

\end{lemma}

%%%%%%%%%%%%%%%%%%%%%%%%%%%%%%%%%%%%%%%%%%%%

\section{Proof of Theorem \ref{mainthm3}}

Suppose $(X_j, \omega_j)$ be a sequence in $\mathcal{K}(3, D, v)$. The induced $\RCD(-1, 6)$ space is $(X_j, d_{\omega_j}, \mu_{\omega_j})$ that is topologically and holomorphically equivalent to $X_j$ itself by Theorem \ref{mainthm2}. 
 
 Since the Ricci curvature is uniformly bounded above as well,  the singular sets of $(X_j, d_{\omega_j}, \mu_{\omega_j})$ and any iterated tangent cones are closed of Hausdorff dimension no greater than $2$ by Lemma \ref{lemcod42}. Furthermore, there exists a universal constant $\epsilon>0$ such that for any singular point $p\in (X_j, d_{\omega_j})$, $\nu_{X_j}(p) < 1- \epsilon$. By the general compactness theory for non-collapsed RCD spaces, we can assume that after possibly passing to a subsequence, $(X_j, d_{\omega_j}, \mu_{\omega_j})$ converge to a compact $\RCD(-1, 6)$-space $(Z, d_Z, \mu_Z)$. Let $(V, o)$ be a tangent cone at $(Z, p)$. Let  $\cT_V$ be the closure for the set of points $(V, o)$ that limits of $\cS(X_j)$. Then $\cT_V$ is closed and the volume density of each point in $\cT_V$ is less than $1-\epsilon$ for some uniform $\epsilon>0$. In particular, any tangent cone of $\cT_V$ cannot split off $\mathbb{R}^4$ (c.f. Lemma \ref{sinn-3}) and the capacity of $\cT_V$ is $0$. On the other hand, $V\setminus \cT_V$ consists of points that at limits from $\cR(X_j)$ with Ricci curvature uniformly tending to $0$. Therefore $\cS(V) \setminus \cT_V$ has Hausdorff dimension no greater than $2$ and it is closed in $V\setminus \cT_V$. This immediately implies that the singular set $\cS(V)$ must also have capacity $0$. One can apply the standard partial $C^0$-estimates in \cite{DS1} with geometric $L^2$-estimates in Lemma \ref{L242} and Lemma \ref{L241}. The improvement of the partial $C^0$-estimate by \cite{Zk} can also be applied since the Sobolev constant is uniform for $(X_j, \omega_j)$ by the RCD theory. We have completed the proof of Theorem \ref{mainthm3}.

%%%%%%%%%%%%%%%%%%%%%%%%%%%%%%%%%%%%%%%%%%%%

 \section{Examples of $\mathcal{RK}(n)$} 
 
 In this section, we will construct examples in $\mathcal{RK}(n)$. Let $X$ be a $n$-dimensional projective variety with log terminal singularities. We let $\Omega$ be a fixed smooth adapted volume measure on $X$. We choose a smooth K\"ahler metric $\omega_0$ and aim to construct singular K\"ahler metrics in $\mathcal{RK}(X)$ in the K\"ahler class $[\omega_0]$ by  the following complex Monge-Amp\`ere equation
 \begin{equation}\label{exmaeq}
 (\omega_0 + \ddbar \varphi)^n = e^{\lambda \varphi - f } \Omega. 
 \end{equation}

 We define 
 \begin{equation}\label{admsets}
 \cF= \{ f \in  \PSH(X,   \tau \omega_0)\cap C^\infty(X^\circ): ~{\rm for~some~}  \tau>0. \}
 \end{equation}

The following theorem will provide abundant examples in $\mathcal{RK}(X)$ as well as in the RCD theory if $n=3$.

 \begin{theorem}\label{exmthm}
 Suppose $\lambda\geq 0$ and $f\in \cF$. If $e^{-f} \in L^p(X, \omega_0^n)$ for some $p>1$ with $\int_X e^{-f} \Omega = [\omega_0]^n$, then there exists a unique solution $\varphi\in \PSH(X, \omega_0)\cap C^\infty(X^\circ)$ solving equation (\ref{exmaeq}). Furthermore,
 $$\omega =\omega_0+\ddbar \varphi \in \mathcal{RK}(X). $$
 In particular, if $n=3$, the metric measure space $(\hat X, d_\omega, \mu_\omega)$ induced by $(X, \omega)$ is an RCD space homeomorphic to $X$ itself.
 
 \end{theorem}
 
 \begin{proof} Since $f\in \PSH(X, \tau\omega_0)$, we can apply Lemma \ref{pshapp} to approximate $f$ by smooth $f_i\in \PSH(X, \tau\omega_0)$ with $\int_X e^{-f_i}\Omega = \int_X e^{-f}\Omega$. The original equation (\ref{exmaeq}) can be approximated by 
$$(\omega_0 + \ddbar \varphi_i)^n = e^{\lambda \varphi_i - f_i} \Omega.$$
The above approximating equation always admits a unique solution $\varphi_i \in \PSH(X, \omega_0)\cap C^\infty(X^\circ)$. We can assume that $\| \varphi_i\|_{L^\infty(X)}$ is uniformly bounded by the classical $L^\infty$-estimate for complex Monge-Amp\`ere equation \cite{Ko1, EGZ, Zz, GPT}. In particular, if we let $\omega_i = \omega_0 + \ddbar \varphi_i$, we have
\begin{equation}\label{ricexam}
\ric(\omega_i) = \lambda\omega_i + \left( \ddbar f -\ric(\Omega) - \lambda \omega_0 \right) \geq \lambda \omega - A \omega_0
\end{equation}
for some uniform $A>0$ independent of $i$. We can consider the quantity $H_i = \log \tr_{\omega_i} (\omega_0) - B \varphi_i$.  We can apply the maximum principle to $H_i$ for a fixed sufficiently large $B>0$ independent of $i>0$ and derive a uniform upper bound for  $H_i$. Making use of the $L^\infty$-bound for $\varphi_i$ again, $tr_{\omega_i}(\omega_0)$ is uniformly bounded above. Therefore there exists $C>0$ such that for all $i>0$, 
$$\ric(\omega_i) \geq - C \omega_i.$$
The lemma is proved after letting $i\rightarrow \infty$. 
 \end{proof}

 The construction in Theorem \ref{exmthm} can be further simplified by only requiring $f\in C^\infty(X)$ in equation (\ref{exmaeq}).

 \begin{corollary}\label{exmthm2}
 Suppose $\lambda\geq 0$ and $f\in C^\infty(X)$ with $\int_X e^{-f}\Omega = [\omega_0]^n$. Then there exists a unique solution $\varphi\in \PSH(X, \omega_0)\cap C^\infty(X^\circ)$ solving equation (\ref{exmaeq}). Furthermore,
 $$\omega =\omega_0+\ddbar \varphi \in \mathcal{RK}(X). $$
 \end{corollary}
 %

%%%%%%%%%%%%%%%%%%%%%%%%%%%%%%%%%%%%%%%%%%%%

 \section{Compactness of three dimensional singular K\"ahler-Einstein spaces} 
 
 Theorem \ref{mainthm3} can be easily applied to obtain compactness of K\"ahler-Einstein spaces with positive or negative scalar curvature in complex dimension $3$.
 
 \begin{theorem}\label{kecomp1} Let $\mathcal{KE}^+(v)$ be a set of three dimensional K\"ahler-Einstein spaces $(X_j, d_{\omega_j}, \mu_{\omega_j})$ induced by $(X_j, \omega_j) \in \mathcal{RK}(3)$ satisfying 
 \begin{enumerate}
 
 \item $\ric(\omega_j) = \omega_j$, 
 
 \smallskip
 
 \item $\vol(X_j, \omega_j) \geq v$ for some $v>0$ and for all $j>0$.
 
 \end{enumerate}
 Then $$\overline{\mathcal{KE}^+(v)} = \mathcal{KE}^+(v), $$
 where the completion is taken in Gromov-Hausdorff distance.

 \end{theorem}

We can also extend the compactness theorem for K\"ahler-Einstein three-folds of negative scalar curvature in \cite{STW1} to the singular setting in complex dimension $3$.
 
\begin{theorem} \label{kecomp2} Given $v, V>0$, let $(X_j, \omega_j)$ be a sequence of projective varieties of log terminal singularities with $\dim_{\mathbb{C}}X =3$  satisfying
\begin{enumerate}
%\item $X_j$  is a normal projective variety with log terminal singularities, 
%\medskip
%
\item $\ric(\omega_j) = - \omega_j$,
\medskip
\item $v\leq {\rm Vol}(X_i, \omega_i) \leq V$.
\end{enumerate}
After passing to a subsequence, there exist $m\in \mathbb{Z}^+$ and a sequence of $P_j = (p_{1, j}, p_{2,j}, ..., p_{m, j}) \in \amalg_{k=1}^m   X_{{j}}$, such that the pointed RCD spaces $(X_j,  P_j, d_{\omega_j}, \mu_{\omega_j})$ induced by $(X_j, \omega_j)$ converge in the pointed Gromov-Hausdorff topology to a finite disjoint union of complete $\RCD(-1, 6)$ spaces
\begin{equation}
(\mathcal{Y}, d_{\mathcal{Y}})  = \amalg_{k=1}^m (\mathcal{Y}_k, d_k)
\end{equation}
satisfying the following.

\begin{enumerate}

\item For each $k$, $\left\{ (X_j,  \omega_j, p_{k, j}) \right\}_{j=1}^\infty$ converges smoothly to a $3$-dimensional K\"ahler-Einstein manifold $(\mathcal{Y}_k\setminus \mathcal{S}(\mathcal{Y}_k), \mathpzc{g}_k)$ away from  $\mathcal{S}(\mathcal{Y}_k)$.  In particular, $\mathcal{S}_k$ is closed of Hausdorff dimension no greater than $2$.

\smallskip

\item  $(\mathcal{Y}_k, \mathcal{J}_k)$ is homeomorphic to $3$-dimensional  normal quasi-projective variety with at worst log terminal singularities.

\smallskip

\item $\sum_{k=1}^m \textnormal{Vol}(\mathcal{Y}_k, d_k)= \lim_{j\rightarrow \infty} \textnormal{Vol}(X_j, \omega_j)$.

\smallskip

\item There exists a unique projective compactification $\overline{\mathcal{Y}}$ of $\mathcal{Y}$ such that $\overline{\mathcal{Y}}$ is a semi-log canonical model and $\overline{\mathcal{Y}}\setminus \mathcal{Y}$ is the non-log-terminal locus of $\overline{\mathcal{Y}}$.

\end{enumerate}

\end{theorem}

The results in \cite{STW1, STW2} can also be extended to the moduli space of three dimensional canonical models with log terminal singularities. The Weil-Petersson metric can be extended to the KSBA compactification of the above moduli space with continuous K\"ahler potential and has finite volume as a rational number.

%%%%%%%%%%%%%%%%%%%%%%%%%%%%%%%%%

 \bigskip
 \bigskip
 
\noindent {\bf{Acknowledgements:}} The authors would like to thank G\'abor Sz\'ekleyhidi for inspiring discussions and for bringing the work of Honda \cite{Ho} to our attention, which  simplifies some of the argument in the earlier version of this paper. Theorem \ref{mainthm2} was announced at the Brin workshop on 'Several complex variables, complex geometry and related PDEs' in June, 2024. Part of the work was carried out during the AIM workshop on 'PDE methods in complex geometry' in August, 2024. The second and third named authors would like to thank the hospitality of American Institute of Mathematics.

%%%%%%%%%%%%%%%%%%%%%%%%%%%%%%%%%%%%%%%%%%%%

\end{document}